%% file: manuscript.tex
\documentclass[review,hidelinks,onefignum,onetabnum]{siamart220329}

\nolinenumbers

\input{ex_shared}

\ifpdf
\hypersetup{
  pdftitle={Alternating steepest descent methods for tensor completion with applications to spectromicroscopy},
  pdfauthor={O. Townsend, S. Dolgov, S. Gazzola and M. Kilmer}
}
\fi

\usepackage{notation}
\usepackage{comment}
\usepackage{euscript}
\usepackage{subcaption}
\usepackage{dsfont}

\newcommand{\starM}{\star_{\mbox{\scriptsize{M}}}}
\begin{document}

\maketitle

\begin{abstract}
In this paper we develop two new Tensor Alternating Steepest Descent algorithms for tensor completion in the low-rank $\starM$-product format, whereby we aim to reconstruct an entire low-rank tensor from a small number of measurements thereof. Both algorithms are rooted in the Alternating Steepest Descent (ASD) method for matrix completion, first proposed in [J. Tanner and K. Wei, \emph{Appl. Comput. Harmon. Anal.}, 40 (2016), pp. 417-429]. In deriving the new methods we target the X-ray spectromicroscopy undersampling problem, whereby data are collected by scanning a specimen on a rectangular viewpoint with X-ray beams of different energies. The recorded absorptions coefficients of the mixed specimen materials are naturally stored in a third-order tensor, with spatial horizontal and vertical axes, and an energy axis. To speed the X-ray spectromicroscopy measurement process up, only a fraction of tubes from (a reshaped version of) this tensor are fully scanned, leading to a tensor completion problem. 
In this framework we can apply any transform (such as the Fourier transform) to the tensor tube by tube, providing a natural way to work with the $\starM$-tensor algebra, and propose: (1) a tensor completion algorithm that is essentially ASD reformulated in the $\starM$-induced metric space and (2) a tensor completion algorithm that solves a set of (readily parallelizable) independent matrix completion problems for the frontal slices of the transformed tensor. 
The two new methods are 
tested on 
real X-ray spectromicroscopy data, demonstrating that they achieve the same reconstruction error with fewer samples from the tensor compared to the matrix completion algorithms applied to a flattened tensor. 
\end{abstract}



\section{Introduction}
\label{sec:intro}

Matrix and tensor completion problems aim at recovering (or approximating) all entries of an array (matrix or tensor) from a (very) small portion of its entries. Matrix and tensor completion problems appear in a variety of applications such as (just to name a few): recommender systems (including the so-called `Netflix problem') \cite{reccomender}, triangulation from incomplete data \cite{triangulation}, image and video data denoising and inpainting \cite{RPCP, Zhang1}, hyperspectral imaging \cite{Ely}. In this paper we are particularly concerned with applications in undersampled X-ray spectromicroscopy (or, briefly, spectromicroscopy), whereby one wants to infer the chemical composition of a specimen by scanning a (much) reduced set of all possible combinations of spatial pixels and energies. The two main, interlinked, reasons for undersampling spectromicroscopy data are to speed up the measurement process and to mitigate degradation of the specimen due to high X-ray dose.

Formulated this way, matrix and tensor completion are ill-posed inverse problems. To resolve them, the usual approach is to assume that the sought matrix or tensor has low rank(s). For tensors, there exist different definitions of rank.
The most direct extension of the matrix low-rank (skeleton) decomposition to tensors is
the so-called Canonical Polyadic decomposition \cite{CPD}, in which the tensor is decomposed into a sum of rank one tensors, where a rank one tensor is an outer product of fibers (i.e. vectors oriented in each of the different dimensions). 
But there are many other tensor decompostions, and each has its own 
notion of rank.  Some of the other well-known tensor decompositions include 
Tucker \cite{tucker},
Tensor-Train \cite{TT},
Hierarchical Tucker \cite{hk-ht-2009,gras-hsvd-2010},
and $\starM$-product rank \cite{mstar}. 
Which decomposition to use is highly application specific.  

In this work, we make use of low tensor $\starM$ product ranks, which we will show to be particularly well suited to the problem of completion in X-ray spectromicroscopy.
To motivate our approach, we note that the current state-of-the-art pipeline, shown in \cref{fig:schematics}, involves storing data in a 3D-tensor, before flattening it to apply matrix analysis techniques such as Principal Component Analysis (PCA) and K-Means clustering; see, e.g., \cite{LEROTIC200435}.  
This data tensor is usually rather large.
XANES experiments, for example, are of the order of $150\times 100\times 100$;
correspondingly, the flattened matrix is of size $150\times 10000$.
Existing undersampling approaches involve applying matrix completion algorithms to the flattened matrix. 
\begin{figure}
    \centering
    \includegraphics[width=0.75\textwidth]{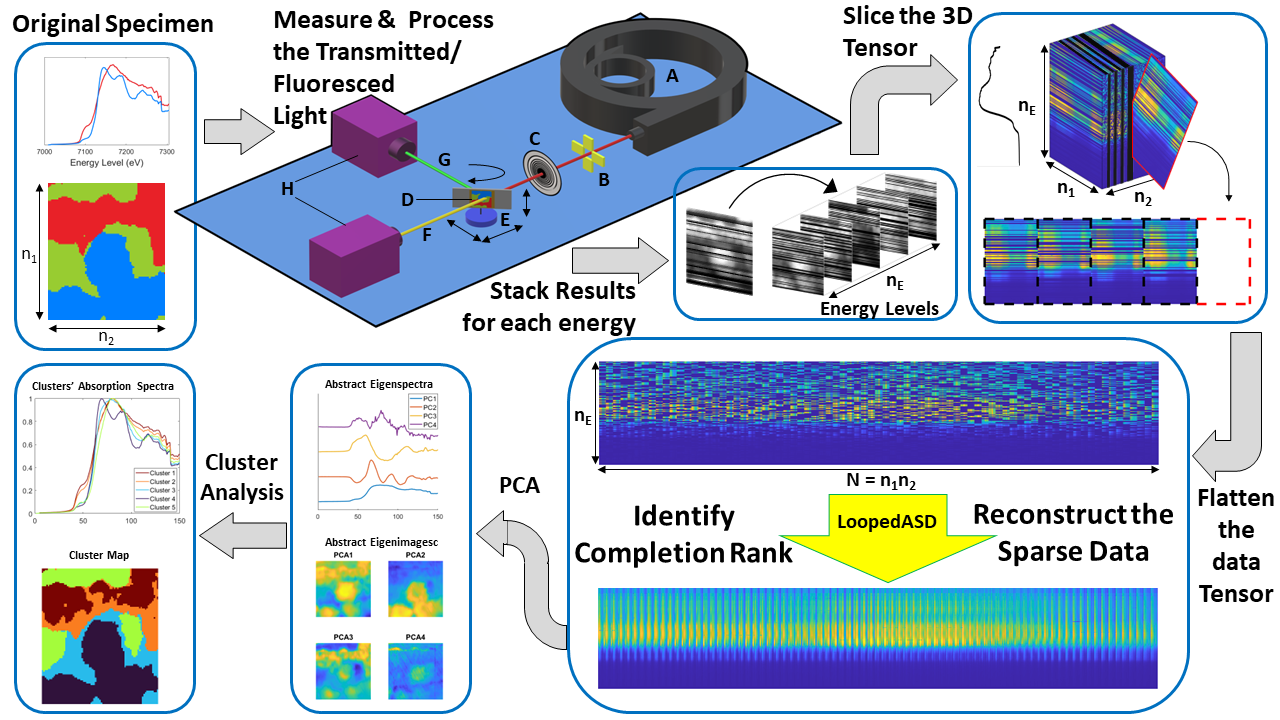}
    \caption{Schematics illustrating the pipeline for X-ray spectromicroscopy data processing.}
    \label{fig:schematics}
\end{figure}
An informal reasoning why X-ray spectromicroscopy data are particularly suitable for low rank matrix completion is that the specimen contains only a few materials exhibiting significantly different spectral responses.
Without noise, each column of a tensor would be a linear combination of those responses, making the whole tensor low-rank.
Since the real data is noisy (due to the instrument error, specimen drift and degradation due to, e.g., changing temperature), the tensor is only approximately low-rank.
Nevertheless, \cite{undersampling} demonstrated that the X-ray spectromicroscopy data are indeed low rank up to a small error.

Matrix (and tensor) completion may depend significantly on the distribution of the known entries (sampling pattern).
One of the most studied (and usually successful) distributions for matrix completion is the Bernoulli distribution.
However, in spectromicroscopy it is more efficient to raster scan the specimen, where the scanner is moved continuously across each row of the specimen.
This, and other reasons, make the $\starM$-product decomposition particularly suitable for spectromicroscopy. Indeed, first, without loss of generality, we can align the tube fibres (i.e. the direction over which we apply the transform) with the direction of the raster sampling. This guarantees that for each known entry we also know the entire tube. Thus, we can apply any transform (e.g. the regular FFT) to each of the known tubes. Second, each frontal slice is more `square' compared to the flattened tensor, which usually admits lower undersampling ratios, although this effect may be offset slightly since the frontal slices are smaller overall. Third, in this setting, the indices of known entries in each frontal slice are sampled from the Bernoulli distribution (in contrast to a raster sampling pattern in the flattened tensor), which may again admit lower undersampling ratios. Finally, suitable transforms will compress the data, so that fewer components need to be recovered for an accurate completion.

As noted above, there are many different types of tensor decompositions and the specification of the tensor completion problem depends on the decomposition chosen.  While some decompositions such as CP and Tucker are orientation independent (i.e. the dimensions are not preferentially treated, so the decomposition is fundamentally the same under permutation of the indices), the $\starM$-product decomposition is orientation dependent.  Significant compressibility of the tensor is possible under the $\starM$ framework when the data can be arranged to take advantage of the orientation dependence.  
Under the $\starM$ framework, tensors can be decomposed in a natural way with provable Eckart–Young optimality with the approximation error not greater than the matrix approximation error \cite{mstar}.  
Unlike other tensor decomposition methods, the $\starM$ framework in the third-order tensor case allows the generalization of all classical algorithms from linear algebra, such as a direct analogue of a singular value decomposition for tensors. 

Many low rank matrix completion algorithms have been published in recent years. These algorithms make use of a range of techniques to achieve accurate recovery of the data, including nuclear norm minimisation \cite{svt} and alternating steepest descent \cite{ASD}. Indeed, we found the latter to perform the best with real data, which lead to the development of LoopedASD \cite{undersampling}, outlined in \cref{sect: loopedASD}.
We will develop our new tensor completion algorithm TASDII based on LoopedASD,
and demonstrate that it is actually more accurate on certain data sets.

The paper is organised as follows. In \cref{sec:background} we report some background material, including mathematical formulations of the matrix and tensor completion problems, relevant methods for matrix completion and the tensor framework adopted in the paper. Our main contributions are presented in \cref{sec:main}, where we present our tensor completion algorithms TASD and TASDII. Experimental results follow in \cref{sec:experiments}, and finally we write our conclusions in \cref{sec:conclusions}. Throughout  the paper we explain relevant notations as soon as they are introduced;  here we just clarify that we use MATLAB-like notation for entries of third-order tensors where, in particular, $A_{i ,j ,k}$ denotes the entry at row $i$ and column $j$ of the matrix going $k$ ``inward''.

\section{Background}
\label{sec:background}
In this section we provide some necessary background by formally stating the matrix and tensor completion problems, by discussing a couple of alternating steepest descent algorithms for matrix completion, and by reviewing the $\starM$ framework: a known matrix-mimetic tensor-algebraic framework whereby the two new alternating steepest descent algorithms for tensor completion will be formulated. 

\subsection{Matrix and tensor completion}

Considering the matrix framework first, let $A \in \C^{n_1\x N}$ be a matrix of low rank $\widetilde{r} \ll \min(n_1,N)$. 
Let $\Omega \subset [n_1] \x [N]:=\{1,2,\dots,n_1\}\x\{1,2,\dots,N\}$ be a set of indices, also referred to as the \textit{sampling pattern}, for which we define the \textit{sampling operator} $\sampleOP:\C^{n_1\x N} \rightarrow \C^{n_1\x N}$, such that
\begin{equation}
\sample{A} =
    \begin{cases}
     A_{ij}           & \text{if }\quad (i,j) \in \Omega,\\
    0                 & \text{otherwise}.\\
    \end{cases}
\end{equation}
Note that $\sampleOP$ is the orthogonal projection on the space supported by $\Omega$. We denote the acquired (sparse) samples of $A$ as $D = \sample{A}$. 
The standard low rank matrix completion problem involves computing the matrix of minimal rank such that it is equal to $D$ when sampled on $\Omega$. In formulas: given a (sparse) sample $D$ of $A$, compute
\begin{equation} \label{eq:LRCP}
    Z^* = \argmin_{Z\in \C^{m\x n}} \ \text{rank}(Z),\qquad \text{subject to} \qquad  \mathcal{P}_\Omega (Z) = D.
\end{equation}
This problem is non-convex and NP hard, so practical approaches estimate its solution by imposing some additional constraints: this is the case of the alternating steepest descent (ASD) algorithm \cite{ASD}, which is summarized in the next subsection and which is one of the main building blocks of the new methods derived in this paper. Other algorithms to approximate the solution to \eqref{eq:LRCP} consider some convex relaxations: this is the case for \cite{svt,RPCP,convex_matrix_completion,SVP}. It can be proved theoretically that, with a sufficient number of samples and possible under some additional assumptions on the involved matrices, these approaches manage to compute a good approximation of $Z^\ast$ and, in turn, of $A$. In practice, exact low rank matrices are rare, so often we are considering matrices that are approximately low rank, or have low numerical rank or $\epsilon$-rank, and whose entries may be corrupted by noise. 

Switching now to the tensor framework, let $\euA \in \C^{n_1\x n_2\x n_3}$ be a third-order tensor of low `rank' $t$ (the notions of tensor rank considered here will be made precise in \cref{def:tensorranks}). Let $\Omega \subset [n_1]\x[n_2]\x[n_3]$ be a set of indices, and define the \emph{tensor sampling operator} $\mathcal{P}_\Omega:\C^{n_1\x n_2\x n_3} \rightarrow \C^{n_1\x n_2\x n_3}$ as the orthogonal projection onto the space supported by $\Omega$, i.e.,
\begin{equation}
    [\sample{\euA}]_{ijk} = 
    \begin{cases}
        \euA_{ijk} & \text{if} \quad (i,j,k) \in \Omega,\\
        0 & \text{otherwise}.
    \end{cases}
\end{equation}
Similar to the matrix case, we denote the acquired (sparse) samples of $\euA$ as $\euD = \sample{\euA}$,  i.e., we only have access to entries $\euD = \sample{\euA}$. The basic low rank tensor
completion problem is formulated as follows: Given a (sparse) sample $\euD$ of $\euA$, compute
\begin{equation}
    \euZ^*=\argmin_{\euZ \in \C^{n_1 \x n_2 \x n_3}} \text{`rank'}({\euZ}),\qquad \text{subject to}\qquad \sample{\euZ} = \euD.
\end{equation}

One of the key parameters for the analysis of, and the performance of algorithms for, matrix and tensor completion is the \emph{undersampling ratio} $p$, 
defined as
\begin{equation}
    p = \frac{|\Omega|}{n_1 N}\quad \text{for matrices}\quad\text{and}\quad p = \frac{|\Omega|}{n_1 n_2 n_3}\quad \text{for tensors},
\end{equation}
where $|\Omega|$ denotes the cardinality of the set $\Omega$. 

$\Omega$ is commonly chosen at random. Most of the literature (theory and algorithms) for matrix completion or tensor completion considers 
$\Omega$ to be Bernoulli sampled, whereby one picks each tuple $(i,j)$ (in the matrix setting) or $(i,j,k)$ (in the tensor setting) with probability $p$, leading to an expected undersampling ratio $p$. 

However, in spectromicroscopy and other imaging applications, when undersampling is performed to speed the measurement process up or when it is otherwise necessary, the samples should typically obey some 
constraints dictated by the functioning of the data acquisition machinery. Therefore one should realistically deal with random structured undersampling. In this paper we consider the so-called \emph{raster samples}. Namely, in the natural third-order representation for spectromicroscopy (and possibly other imaging) data, the sampling pattern $\Omega$ is defined by varying $i\in[n_1]$ and $k\in[n_3]$, and by taking $(i,j,k)=1$ with probability $p$ for all $j\in [n_2]$. Although the expected undersampling ratio for such raster scans is $p$, it can be shown that the probability of the sampled tensor to have $(i,j,k)=0$ for all $i\in [n_1]$ and $j\in [n_2]$ increases dramatically with respect to Bernoulli sampling at lower undersampling ratios. Since, in these situations, we have a zero horizontal slice in the tensor data (or a zero row in the matrix representation of it -- see \cref{sec: tensors} for more details about the relations between these two formats), which implies that  the associated tensor (or matrix completion) problem is impossible to solve, 
we consider instead a so-called \emph{robust raster sampling}. Namely, while varying $i\in [n_1]$ and $k\in [n_3]$, we record which indices $i_{\text{sample}}$ have been sampled (i.e., $(i_{\text{sample}},j,k)=1$ for all $j\in [n_2]$) and do not allow the same $i_{\text{sample}}$ to be sampled again until all $i\in [n_1]$ have been sampled. 
This sampling pattern also ensures that measurements are more spread out overall.

\subsection{ASD and LoopedASD for matrix completion}\label{sec:ASD}
The Alternating Steepest Descent (ASD) algorithm \cite{ASD} approximates the solution of \cref{eq:LRCP} by explicitly imposing the iterates to be of the form $Z = XY$, with $X \in \C^{n_1\x r}$ and $Y \in \C^{r\x N}$ of full rank $r\ll\min\{n_1,N\}$, so that they belong to the manifold of rank $r$ matrices, for a fixed $r$. In this setting, ASD solves
\begin{equation}\label{eq:completion problem}
    \min_{Z \in \C^{m \x n}}f(X,Y),\quad\mbox{where}\quad f(X,Y)=\frac{1}{2}||D - \sample{XY}||^2_F\,,
\end{equation}
by applying steepest descent with exact line search alternately to the factors $X$ and $Y$. Writing $f(X,Y)$ as $f_X(Y)$ for fixed $X$ and $f_Y(X)$ for fixed $Y$, we can explicitly compute the gradients $\nabla f_Y(X),\ \nabla f_X(Y)$, and steepest descent stepsizes $\eta_X,\ \eta_Y$ as
\begin{alignat}{3}
    \nabla f_Y(X) &= -(D - \sample{XY})Y^{\rm H}, \quad && \text{and} \quad   \nabla f_X(Y)&&= -X^{\rm H}(D - \sample{XY}),\label{eq:ASD steps1}\\
    \eta_X &= \frac{||\nabla f_Y(X)||^2_F}{||\sample{\nabla f_Y(X)Y}||^2_F}, && \text{and} \ \quad\  \eta_Y&&= \frac{||\nabla f_X(Y)||^2_F}{||\sample{X\nabla f_X(Y)}||^2_F},\label{eq:ASD steps2}
\end{alignat}
where we use $X^{\rm H}$ to denote the Hermitian transpose. 
Using the formulas above, and starting from initial guesses $X_0 \in \C^{n_1\x r}$ and $Y_0 \in \C^{r\x N}$ for the two factors of $Z$, the $(i+1)$th iteration of ASD, $i=0,1,2,\dots$ consists of the following basic steps:
\begin{equation}
    \begin{cases}
        \text{Fixing }\ Y_i,\ \text{ compute }\ \nabla f_{Y_i}(X_i)\ \text{ and }\ \eta_{X_{i}}\\
        \text{Update }\ X_{i+1} = X_{i} - \eta_{X_{i}}\nabla f_{Y_i}(X_i)\\
        \text{Fixing}\ X_{i+1},\ \text{ compute }\ \nabla f_{X_{i+1}}(Y_i)\ \text{ and }\ \eta_{Y_{i}}\\
        \text{Update }Y_{i+1} = Y_{i} - \eta_{Y_{i}}\nabla f_{X_{i+1}}(Y_i).
    \end{cases}\label{eq:ASD}
\end{equation}

We stop ASD at the $i_{\text{stop}}$th iteration if one of the following conditions is satisfied: (1) $i_{\text{stop}}$ equals the maximum number of allowed iterations; (2) a tolerance for the relative residual is reached norm; (3) a tolerance on the difference between the residual at the $i_{\text{stop}}$th iteration and the residual at the $(i_{\text{stop}}-50)$th iteration is reached (so that the algorithm is stopped when the convergence rate slows to nearly zero). Conditions (2) and (3) are cheap to evaluate using the following update formula for the residuals
\begin{align}\label{eq:res update}
    D - \sample{X_{i+1} Y_i}
    = \big(D - \sample{X_i Y_i}\big) +\ \eta_{X_i} \sample{\nabla f_{Y_i}(X_i)Y_i}.
\end{align}
We denote the ASD outputs as $X^\ast$ and $Y^\ast$. 

ASD was chosen because it is particularly quick and efficient. When applied optimally, it has a per-iteration cost of only $8|\Omega|r$ and was found to be much more effective when completing larger, noisy data sets. For the purpose of completing real, physical data sets with a low $\epsilon$-rank, ASD performed better by computing accurate approximations while limiting the complexity of the data. 

Despite this, there are some implementation issues associated with ASD. In particular, one must choose an appropriate completion rank $r$ to obtain the best results. This parameter limits the size of the factors, thus if $r$ is set too low the true variation of the data can not be captured. Setting $r$ too high will result in longer computation times and overfitting of noisy components. In addition, ASD requires the initial factors to be scaled appropriately, and will struggle to accurately complete data whose completion subspace has a large condition number.

LoopedASD was proposed in \cite{undersampling} as an ASD variant equipped with a rank estimation procedure that enables the algorithm to be implemented in an automatic fashion. With respect to the original LoopedASD version in \cite{undersampling}, here we consider also appropriate scaling and orthogonalization of the (intermediate) low-rank factors (for estimating the final completion rank and for performing the final completion step, respectively), resulting in a reliable, robust and accurate algorithm that avoids convergence to spurious local minima and outperforms ASD even when completing poorly scaled sparse matrices, data with high condition numbers, and particularly at lower undersampling ratios. More details about LoopedASD are given in appendix \cref{sect: loopedASD}.

\subsection{Tensor framework: $\starM$-product, t-SVDM, tensor ranks and inner products}\label{sec: tensors} 

In this paper we will use MATLAB notation to describe the entries, vectors and slices taken from tensors, as well as the reshaping of tensors into matrices and vice versa. For a tensor $\euA \in \C^{n_1 \x n_2\x n_3}$, we denote the entry on the $i^{th}$ row, the $j^{th}$ column, and $k^{th}$ matrix back by $\euA_{i,j,k} \in \C$. Of particular importance are the \emph{tube fibres}, $\euA_{i,j,:}\in \C^{n_3}$, the vectors defined by fixing the first two indices. By fixing just one index one can define the following matrices: the \emph{horizontal slices}, $\euA_{i,:,:}\in \C^{n_2\x n_3}$, the \emph{lateral slices}, $\euA_{:,j,:}\in \C^{n_1\x n_3}$, and the \emph{frontal slices}, $\euA_{:,:,k} \in \C^{n_1\x n_2}$.

On occasion, we will reformat datasets from tensors to matrices and from matrices back to tensors. The \emph{reshape}$(A,sz)$ function restructures the data as in MATLAB, so that the output will have size $sz$. When flattening a tensor, we stick each frontal slice to the right of the previous one; one can also consider vectorising each horizontal slice to create the corresponding row of the matrix. For example, for $\euA \in \C^{n_1 \x n_2\x n_3}$, we define the matrix
\begin{equation}\label{eq:reshape}
    \reshape(\euA,[n_1,n_2n_3]) = A \in \C^{n_1\x n_2n_3},
\end{equation}
where
\begin{equation}
    A_{i,:} = \text{vec}(\euA_{i,:,:}),\quad \forall i = 1,...,n_1 
\end{equation}
where $vec$ vectorises the horizontal slices of $\euA$. The reverse process will reshape the matrix into a tensor, i.e.
\begin{equation}\label{eq:flatten}
    \reshape(A,[n_1,n_2,n_3]) = \euA \in \C^{n_1\x n_2\x n_3},
\end{equation}

We can also map between $n_1\x n_3$ matrices and lateral slices, $\euA_{:,j,:}$, by twisting and squeezing so that for $A \in \C^{n_1 \x n_3}$
\begin{equation}
    \euA = \text{twist}(A) \qquad \text{and} \qquad A = \text{sq}(\euA)
\end{equation}

Defined in \cite{mstar}, the $\starM$-product acts as a generalisation of the t-product seen in \cite{tprod, third_order_tensors_as_operators} by substituting the Fourier transform for any invertible (later restricted to unitary) matrix, $M \in \C^{n_3\x n_3}$. The core idea of this framework is that tensor operations reduce to applying the equivalent matrix operations on the frontal slices in the transform domain, hence `matrix mimetic'. Details on other useful results can be found in the Appendix.

\begin{definition}[Tensor transform]
    For tensor $\euA \in \C^{n_1\x n_2\x n_3}$, we define the mode-3 unfolding as,
    \begin{equation}
        \euA_{(3)} = \left[sq(\euA_{:,1,:})^T,...,sq(\euA_{:,p,:})^T \right]
    \end{equation}
    Essentially, this unfolding reshapes the tensor $\euA$ so that the tube fibres are positioned as the columns of the unfolded matrix.

    Let $M \in \C^{n_3\x n_3}$ be an invertible matrix. We define the \emph{transformed tensor} as,
    \begin{equation}
        \hat{\euA} = \euA \x_3 M = \reshape(M \euA_{(3)},[n_1,n_2,n_3]).
    \end{equation}
\end{definition}    

We have effectively applied the transform matrix $M$ along each of the (3-mode) tube fibres, before reshaping the resulting matrix back into a tensor. We will denote the transform with a circumflex.

\begin{definition}[$\starM$-product]
\label{def:Mprod}
    For $\euA \in \C^{n_1\x n_2\x n_3},\ \euB \in \C^{n_2\x n_4\x n_3}$, and transform matrix $M \in \C^{n_3\x n_3}$, the $\starM$-product $\euC = \euA \starM \euB$ is of size $\euC\in \C^{n_1\x n_4\x n_3}$, and each frontal slice is defined in the transform domain,
    \begin{equation}
        \hat{\euC}_{:,:,k} = \hat{\euA}_{:,:,k}\; \hat{\euB}_{:,:,k},
    \end{equation}
    for $k = 1,...,n_3$.
\end{definition}

Using the $\starM$ framework, one can define many useful properties.
\begin{definition}[Tensor Conjugate Transpose]
    For $\euA \in \C^{n_1\x n_2\x n_3}$, we obtain its conjugate transpose $\euA^{\rm H} \in \C^{n_2\x n_1\x n_3}$ by taking the Hermitian of each frontal slice in the transform domain,
    \begin{equation}
        (\hat{\euA}^{\rm H})_{:,:,k} = (\hat{\euA}_{:,:,k})^{\rm H}
    \end{equation}  
\end{definition}
\begin{definition}[Identity Tensor]
    Similarly the identity tensor is defined,
    \begin{equation}
        (\euI \x_3 M)_{:,:,k} = \hat{\euI}_{:,:,k} = I,\quad \forall k \in [n_3]
    \end{equation}
\end{definition}

\begin{definition}[Orthogonal Tensor]
    A tensor $\euQ \in \C^{m \x m \x n}$ is $\starM$-unitary if 
    \begin{equation}
        \euQ^{\rm H} \starM \euQ = \euI = \euQ \starM \euQ^{\rm H}.
    \end{equation}
\end{definition}

Indeed, the frobenius norm has a very natural extension to tensors,
\begin{equation}
    \normf{\euA}^2 = \sum_{ijk}|\euA_{ijk}|^2.
\end{equation}

In \cite{mstar} it is shown that under the $\starM$ framework (and certain conditions on $M$), $\starM$-unitary tensors satisfy similar properties to unitary matrices. In particular, the frobenius norm is invariant under $\starM$-multiplication with orthogonal tensors. 

From the next result onwards, we restrict the transform matrix $M$ to scalar multiples of unitary matrices, 
\begin{equation}\label{eq:Morth}
    M = cW\ \text{for}\ c \in \C\backslash\{0\}\ \text{and unitary}\ W \in \C^{n_3\x n_3}
\end{equation}

\begin{lemma}\label{lem:unitarytensor}
    For $\starM$-unitary $\euQ \in \C^{n_1\x n_1\x n_3}$, $\euA \in \C^{n_1\x n_2 \x n_3}$, and \linebreak[4]$\euB \in \C^{n_2\x n_1 \x n_3}$,
    \begin{equation}
        \normf{\euQ \starM \euA} = \normf{\euA}\quad \text{and}\quad \normf{\euB \starM \euQ} = \normf{\euB},
    \end{equation}
\end{lemma}
The proof of \cref{lem:unitarytensor} can be found in \cite{mstar}.

We now construct a tensor inner product and inner product space, so that we can derive the gradients and step-sizes for our first tensor algorithm, TASD. 

\begin{definition}[Tensor Trace]
    For $\euA \in \C^{n_1\x n_1\x n_3}$, the tensor trace, $\tTr:\C^{n_1\x n_1 \x n_3}$, is defined 
    \begin{equation}
        \tTr(A) = \sum^{n_3}_{k=1} tr(\hat{\euA}_{:,:,k})
    \end{equation}
\end{definition}

\begin{definition}[Tensor Inner Product]\label{def:TIP}
    For $\euA,\ \euB \in \C^{n_1\x n_2\x n_3}$, we define the Tensor Inner Product (TIP) as 
    \begin{equation}
        \langle \cdot,\cdot\rangle: \C^{n_1\x n_2\x n_3} \x \C^{n_1\x n_2\x n_3} \rightarrow \C, \quad \langle \euA, \euB\rangle = tTr(\euA \starM \euB^{\rm H}).
    \end{equation}
\end{definition}
\begin{lemma}\label{lem:TIP}
    The TIP is an inner product, under which we define a tensor inner product space. Further, it can be shown that for unitary $M$,
    \begin{equation}
        \langle \euA, \euA \rangle = \normf{\euA}^2.
    \end{equation}
\end{lemma}
In the frequent case that our transform matrix is a scalar product of a unitary matrix, then the inner product will need to be scaled appropriately for this identity to hold. The proof of \cref{lem:TIP} is given in the appendix.

Our final example of a generalised tensor function using the $\starM$ framework is the t-SVDM. The t-SVD was originally proposed in \cite{tprod} under the t-product framework, but was been generalised to the $\starM$ framework in \cite{mstar}, where a more detailed description and the pseudocode can be found. The t-SVDM naturally leads to the definition of (several) tensor ranks, which will be used extensively throughout our work on low rank tensor completion. 

\begin{theorem}[t-SVDM]\label{thm:t-rank}
    For $\euA \in \C^{n_1 \x n_2 \x n_3}$, the t-SVDM of $\euA$ is given by 
    \begin{equation}
        \euA = \euU \starM \euS \starM \euV^{\rm H},
    \end{equation}
where $\euU \in \C^{n_1 \x n_1 \x n_3}$ and $\euV\in \C^{n_2\x n_2\x n_3}$ are $\starM$-unitary tensors, and $\euS\in\C^{n_1 \x n_2\x n_3}$ is an f-diagonal tensor. It is computed by decomposing each frontal slice of $\hat{\euA}$ using the matrix SVD.
\end{theorem}

\begin{definition}[Tensor ranks]\label{def:tensorranks}
Using the t-SVDM, there are several ranks that can be defined for tensors. For $\euA \in \C^{n_1\x n_2\x n_3}$:\\
The \textbf{t-rank} is the number of non-zero tube fibres in $\euS$.\\
The \textbf{multirank} is the vector $\rho \in \R^{n}$, such that $\rho_k = \rank(\hat{\euA}_{:,:,k})$. \\
The \textbf{implicit rank} is the sum of the multirank $r = \sum_k \rho_k$.
\end{definition}
When evaluating the performance of our tensor completion algorithms against ASD and LoopedASD, it is useful to be able to compare between matrix ranks and tensor t-ranks. The following result describes the relationship between these properties; the proof of which can also be found in \cite{mstar}. 
\begin{theorem}\label{thm:trankbound}
    Given tensor $\euA \in \C^{n_1\x n_2 \x n_3}$ and matrix $A \in \C^{n_1 \x n_2n_3}$ such that $\euA = \reshape (A,[n_1,n_2,n_3])$ (according to the method of reshaping set out in equation \ref{eq:reshape}). Then the t-rank, $t$, of $\euA$ is bounded by the rank, $r$, of $A$; i.e. $t\leq r$.
\end{theorem}

\section{Tensor Alternating Steepest Descent algorithms}
\label{sec:main}

In \cref{sec:ASD}, we introduced the matrix completion algorithms ASD and LoopedASD, and in \cite{undersampling} it was demonstrated that these algorithms are effective at recovering sparse hyper spectral data such as x-ray spectromicroscopy. In many such cases, the data sets are originally 3 dimensional and require reshaping to be compatible with the completion algorithms. We therefore seek to develop tensor completion algorithms that would allow the recovery of sparse 3D data without the need for flattening.

Similar to matrices, there are many approaches one can take to achieve the completion of sparse tensors. Continuing the precedent set by our work on matrix completion, we focus on alternating least squares approaches. Indeed, our first algorithm, TASD in Section~\ref{sec:TASD}, is a direct 3D analogue to ASD. Our second algorithm, TASDII in Section~\ref{sec:TASDII}, makes extensive use of LoopedASD within the $\starM$-product framework.

\subsection{TASD}
\label{sec:TASD}
Similar to ASD, we first restrict the iterates to the manifold of tensors with t-rank $t$ by imposing the following decomposition on $\euZ \in \C ^{n_1 \x n_2 \x n_3}$,
\begin{equation}
    \euZ = \euX \starM \euY, \quad \text{for}\ \euX \in \C^{n_1\x t \x n_3},\ \euY \in \C^{t\x n_2 \x n_3}.
\end{equation}
We wish to solve,
\begin{equation}\label{eq:objfunc}
    \min_{\euX, \euY} f(\euX,\euY),\quad \text{where}\quad f(\euX,\euY) = \frac{1}{2}||\euD - \sample{\euX \starM \euY}||^2_F
\end{equation}

We now follow the same optimisation steps as ASD, but for tensors under the $\starM$ product. TASD fixes one of the components of $f$ and implements one step of steepest descent with exact line search. The fixed component is alternated on each iteration. We write $f_{\euY}(\euX)$, $f_{\euX}(\euY)$ to denote the function $f$ when fixing $\euY$, $\euX$ respectively.

We compute the gradients as follows,
\begin{align}\label{eq:TASDgrad}
        \nabla f_{\euY} (\euX) &= -(\euD - \sample{\euX \starM \euY })\starM\euY^{\rm H}, \\
        \nabla f_{\euX} (\euY) &= -\euX^{\rm H} \starM (\euD - \sample{\euX \starM \euY }).
\end{align}
A detailed derivation is provided in \cref{sec:TASDderiv}. Notice that this is the same expression as in the matrix case, but replacing the matrix product by the $\starM$ product. 
Similarly we can compute the exact step sizes required to minimise the objective function in the gradient directions,
\begin{align}\label{eq:stepsizes}
        \eta_{\euX} &= \frac{\normf{\nabla f_{\euY} (\euX)}^2}{\normf{\sample{\nabla f_{\euY} (\euX) \starM \euY}}^2} &\eta_{\euY} = \frac{\normf{\nabla f_{\euX}(\euY)}^2}{\normf{\sample{\euX \starM \nabla f_{\euX}(\euY)}}^2}.
\end{align}

Using our expressions for the gradient and step sizes above, we construct the following TASD algorithm:
\begin{equation}
    \begin{cases}
        \text{Fixing }\ \euY_i,\ \text{ compute }\ \nabla f_{\euY_i}(\euX_i)\ \text{ and }\ \eta_{\euX_{i}}\\
        \text{Update }\ \euX_{i+1} = \euX_{i} - \eta_{\euX_{i}}\nabla f_{\euY_i}(\euX_i)\\
        \text{Fixing}\ \euX_{i+1},\ \text{ compute }\ \nabla f_{\euX_{i+1}}(\euY_i)\ \text{ and }\ \eta_{\euY_{i}}\\
        \text{Update }\euY_{i+1} = \euY_{i} - \eta_{\euY_{i}}\nabla f_{\euX_{i+1}}(\euY_i).
    \end{cases}\label{eq:TASDprocedure}
\end{equation}

One approach that helps optimise this implementation is the cheap residual update that can be performed each iteration. As for ASD, a similar update scheme for TASD can be used after updating each component, say $\euX_i$,
\begin{align}\label{eq:resupdate}
    \euD - \sample{\euX_{i+1} \starM \euY_i} &= \euD - \sample{(\euX_i - \eta_{\euX_i} \nabla f_{\euY_i}(\euX_i)) \starM \euY_i} \\
    &= \euD - \sample{\euX_i \starM \euY_i} + \eta_{\euX_i} \sample{\nabla f_{\euY_i}(\euX_i) \starM \euY_i}.
\end{align}
Note that the term $\sample{\nabla f_{\euY_i}(\euX_i) \starM \euY_i}$ is already computed for the denominator of the step size. 

Naturally, due to the non-convex nature of the objective function, it is impossible to prove convergence to the global minimum from arbitrary starting points. However, following a similar proof to that found in \cite{ASD}, it can be shown that any limit point of the iterates of TASD converge to a stationary point.
\begin{theorem}\label{thm:TASD}
    Let $\euX^*$, $\euY^*$ be limit points generated by TASD, then they satisfy,
    \begin{equation}
        \nabla f_{\euY^*}(\euX^*) = 0 \quad \text{and} \quad \nabla f_{\euX^*}(\euY^*) = 0.
    \end{equation}
\end{theorem}
The proof of Theorem \ref{thm:TASD} is given in supplementary material. 

Similar to ASD, the completion t-rank for TASD must be set before hand. A t-rank estimation procedure would have been possible in a potential LoopedTASD algorithm.
However, since it was found that TASD has longer run times than ASD (and more importantly than TASDII), such an algorithm is not benchmarked.

\subsection{TASDII}
\label{sec:TASDII}
TASD is a direct conversion of the ASD algorithm to tensors, using the $\starM$-product to facilitate tensor-tensor multiplication and induce the inner product space we can differentiate over. However, as we have seen many times in Section \ref{sec:background}, the core principal of the $\starM$-framework is to define tensor operations by applying the equivalent matrix operations on the frontal slices in the transform domain. Thus rather than completing tensors as a whole, we now aim to complete the frontal slices in transform domain.

It was shown in \cite{third_order_tensors_as_operators, mstar} that using a varied multirank (different rank per slice) will improve data compression compared to simply restricting the t-rank of the tensors. See the comparison of t-SVDM, which thresholds singular tubes in the product domain, and t-SVDMII, which thresholds all singular values across all the frontal slices in the transform domain. It can be proven that the approximation error of t-SVDMII is bounded by that of t-SVDM for the same implicit rank. 

TASD is comparable to t-SVDM, since we are completing the data in the product domain and the iterates are thresholded by the t-rank. We now propose a novel tensor completion algorithm, TASDII, which completes frontal slices independently in the transform domain and allows for a more flexible multirank, similar to t-SVDMII.

The TASDII algorithm starts from the crucial observation that we have aligned the tube fibres with the raster direction, so for each sample $(i,j,k) \in \Omega$ there exist samples at all other tube indices, $(i,j,k') \in \Omega$ for all $k'=[n_3]$.
Let 
\begin{equation}
    \overline{\Omega} = \{(i,j) \ | \ (i,j,k) \in \Omega\}
\end{equation}
be the sampling pattern for any frontal slice.
Then, we can transform the tensors along the tubes,
and break the tensor completion problem down into solving $n_3$ matrix completion problems for each of the $n_3$ frontal slices:
\begin{align}
        &\min_{\euX, \euY} \frac{1}{2}||\euD - \sample{\euX \starM \euY}||^2_F\\
        \Leftrightarrow &\min_{\hat{\euX}, \hat{\euY}} \frac{1}{2} \sum_{k=1}^{n_3}||\hat{\euD}_{:,:,k} - \mathcal{P}_{\overline{\Omega}}(\hat{\euX}_{:,:,k}\hat{\euY}_{:,:,k})||^2_F\\
        \Leftrightarrow &\min_{\hat X_k,\hat Y_k} \frac{1}{2}\normf{\hat{\euD}_{:,:,k} - \mathcal{P}_{\overline{\Omega}} (\hat X_k \hat Y_k)}^2,
        &  k = 1, ..., n_3.
\end{align}

Each of the $n_3$ frontal slices can be interpreted as an $(n_1 \x n_2)$ matrix with a robust Bernoulli sample pattern, and can be completed independently and in parallel. There is no need for slices to use the same completion rank, nor is it desirable. Depending on the transform matrix, some slices may contribute more towards the completed tensor, while others can be safely ignored. This allows for more time and resources to be put into the more significant slices, improving efficiency.

For example, when completing x-ray spectromicroscopy data using TASDII \linebreak[4]equipped with the t-product (so that $M$ is the DFT), the inner slices correspond to the high frequency components of the data. These high frequency components can be shown to be mostly noise, and do not complete successfully for low undersampling ratios in any case. It is quicker and will produce more accurate completions if these components are ignored and set to zero. A more detailed discussion on the structure of these documents is given in \cref{sec:specstruc}.

It is impractical to manually set the multirank for TASDII, i.e. the completion rank $r_k$ for each frontal slice. Instead, we use LoopedASD to complete each of the slices, using rank estimation. 
In a pseudocode fashion, we write
\begin{align}
    \hat{X}^*_k,\hat{Y}^*_k = \text{LoopedASD}(\hat{\euD}_{:,:,k}),\quad \hat{X}^*_k \in \C^{n_1\x r_k}, \hat{Y}^*_k \in \C^{r_k\x n_2}, \quad k=1,\ldots,n_3.
\end{align}
This will also ensure that TASDII is robust against the scaling and the conditioning (decay of the singular values) of the slices. We allow a rank of $r_k = 0$ when selecting the rank and when thresholding the singular values in the next step, which allows us to make zero slices where completions are not converging.

We now wish to restrict the number of significant components and remove spurious slices from the completion, without affecting the accuracy of the reconstruction. This should remove noisy components, effectively compressing the completion down to its most efficient multirank. This can be done by thresholding the energy of the completed singular values over \emph{all} slices. This thresholding approach is similar to that from t-SVDII algorithm described in \cite{mstar}.

Still in the transform domain, we apply the matrix SVD to each of the completed slices, then collate and sort (in decreasing order) the singular values of all slices into a vector $w$. We compute the cumulative sums of the sorted, squared singular values,
\begin{align}
    \hat{U}_k\hat{S}_k \hat{V}_k^{\rm H}  & = \text{svd}(\hat{X}^*_k\hat{Y}^*_k) \quad \text{then set}\quad w = \text{sort} \Bigg(\bigcup_{k = 1,...,n_3} \mathrm{diag}(\hat{S}_k) \Bigg),\\
    \sigma_j &= \sum^j_{k = 1} w^2_j, \quad \text{and} \quad \hat{\euS}_{:,:,k} = \hat{S}_k.
\end{align}
We then threshold the sums $\sigma_j$ with a relative tolerance $\gamma \in (0,1]$, finding
\begin{equation}
    J = \min j ~|~ \sigma_j > \gamma \cdot ||\hat{\euS}||_F^2.
\end{equation}
Now, all singular values below $w_J$
are set to zero, effectively reducing the rank over many of the slices. Again, note that slices with rank 0 are allowed. Formally,
\begin{equation}
    \hat{\euS}_{i,i,k}  = 0 \quad \text{if}\quad \hat{\euS}_{i,i,k} \le w_{J}.
\end{equation}
Once the multirank has been thresholded, we re-complete compressed slices with a reduced rank, $\rho_k^* = \text{rank}(\hat{\euS}_{:,:,k})$.  Slices with a positive reduced rank will be re-completed using LoopedASD, but now setting the completion rank to the reduced value. To reduce the run time, the previous completion of that slice will be projected down to the new rank manifold and used as a warm start.

From here, one can take full advantage of the flexibility of TASDII, by imposing sparsity patterns or other similar rules across the frontal slices. For example, in the case of x-ray spectromicroscopy it can be shown that non-zero slices entirely surrounded by zero slices are often overfit with a relative error greater than 1. To avoid such slices increasing the completion error, we simply set any slice surrounded by zero slices to zero as well.

Finally, stacking each completed frontal slice and applying the inverse transform will recover our completed tensor:
\begin{align}
    \hat{\euX}_{:,:,k}^* = \hat X_k^*,  \ \ \hat{\euY}_{:,:,k}^* = \hat Y_k^*, & \qquad \hat{\euZ}^*_{:,:,k} = \hat{\euX}_{:,:,k}^*\hat{\euY}_{:,:,k}^*, \quad k = 1, \ldots, n_3,\\
    \euZ^* & = \hat{\euZ}^* \x_3 M^{-1}.
\end{align}
The complete pseudocode for TASDII is provided in Algorithm~\ref{alg:TASDII}.

\begin{algorithm}
\caption{TASDII}
\label{alg:TASDII}
\hspace*{\algorithmicindent} \textbf{Input:} Data tensor $\euD$, truncation threshold $\gamma$, LoopedASD parameters \\
\hspace*{\algorithmicindent} \textbf{Output:} Completed tensor $\euZ^*$
\begin{algorithmic}[1]
\STATE{Set $\hat{\euD} \gets \text{fft}(\euD,[],3)$;}
\FOR{$k = 1:n_3$}
\STATE{$[\hat{\euX}_{:,:,k},\hat{\euY}_{:,:,k}] = \text{LoopedASD}(\hat{\euD}_{:,:,k})$ with automatic rank detection;} 
\STATE{$[\hat{\euU}_{:,:,k},\hat{\euS}_{:,:,k},\hat{\euV}_{:,:,k}] = \text{svd}(\hat{\euX}_{:,:,k}\hat{\euY}_{:,:,k})$;}
\STATE{$\rho_k = \text{rank}(\hat{\euS}_{:,:,k})$;}
\ENDFOR
\STATE{ }
\STATE{Store all f-diagonal entries, $w \gets \bigcup_k \text{diag}\big(\hat{\euS}_{:,:,k}\big)$;}
\STATE{$w \gets \text{sort}(w,\text{`descend'})$;}
\STATE{Compute cumulative sums, $\sigma_j = \textstyle \sum_{k=1}^jw^2_k$;}
\STATE{Set $J = \min j  \ |\ \sigma_j > \gamma ||\hat{\euS}||_F^2$;}
\STATE{Keep only the singular values greater than $w_J$: $\hat{\euS}(\hat{\euS} \le w_J) = 0$;}
\STATE{ }
\FOR{$k = 1:n_3$}
\STATE{$\rho^*_k = \text{rank}(\hat{\euS}_{:,:,k})$;}
\IF{$\rho^*_k = 0$}
\STATE{$\hat{\euX}_{:,:,k} = 0; \quad \hat{\euY}_{:,:,k} = 0$;}
\ENDIF
\IF{$\rho^*_k > 0 \And \rho^*_k < \rho_k$}
\STATE{$[\hat{\euX}_{:,:,k},\hat{\euY}_{:,:,k}] =\text{LoopedASD}(\hat{\euD}_{:,:,k})$ up to the rank $\rho^*_k$;}
\ENDIF
\IF{$\rho^*_k = \rho_k$}
\STATE{$\hat{\euX}_{:,:,k} = \hat{\euU}_{:,:,k}\hat{\euS}^{\frac{1}{2}}_{:,:,k}; \qquad \hat{\euY}_{:,:,k} = \hat{\euS}^{\frac{1}{2}}_{:,:,k} \hat{\euV}^{\rm H}_{:,:,k}$;}
\ENDIF
\ENDFOR
\STATE{ }
\STATE{Apply any additional sparsity pattern across the frontal slices, e.g.}
\FOR{$k = 1:n_3$}
\IF{$\rho^*_{k-1} = 0\ \AND\ \rho^*_{k+1} = 0$}
\STATE{$\hat{\euX}_{:,:,k} = 0; \quad \hat{\euY}_{:,:,k} = 0$;} 
\ENDIF
\STATE{$\hat{\euZ}_{:,:,k} \gets \hat{\euX}_{:,:,k}\hat{\euY}_{:,:,k}$;}
\ENDFOR
\STATE{ }
\STATE{$\euZ^* \gets \text{ifft}(\hat{\euZ},[],3)$\;}
\end{algorithmic}
\end{algorithm}

Since TASDII is essentially a collection of independent LoopedASD completions, 
its convergence properties (which are in turn inherited from ASD) carry over, in particular, it can be shown that TASDII converges to a stationary point.

\section{Numerical Experiments}
\label{sec:experiments}

\subsection{Data}
We consider data sets from the X-ray spectromicroscopy XANES (X-ray Absorption Near Edge Structure) experiments. 
We use 5 such data sets, scanned from 3 different specimen at different resolutions. 
In each case, the specimen was created by making a drop cast mix of iron molecules (Fe$_2$O$_3$ and Fe$_3$O$_4$). 
When taking the measurements for these data sets, a raster sampling pattern was used to reduce data acquisition times. This ensures they are compatible with our algorithms when we align the raster direction with the tube fibres of the tensor to ensure full transforms. 
Dimensions and specimen for each dataset can be found in Table~\ref{tab:spectromicroscopy_datasets}. 

\begin{table}[h!]
    \centering
    \begin{tabular}{c|c|c|c|c|c|c}
        Name & \# Energy Levels & \# Rows & \# Columns & Specimen No.\\
        \hline
        DS1 &  149 & 101 & 101 & 1 \\
        DS2 &  150 & 92  & 79  & 1 \\
        DS3 &  152 & 55  & 54  & 2 \\
        DS4 &  152 & 40  & 40  & 3 \\
        DS5 &  152 & 80  & 80  & 3 \\ 
    \end{tabular}
    \caption{Dimensions of spectromicroscopy datasets}
    \label{tab:spectromicroscopy_datasets}
\end{table}

\subsection{Implementation}
The key metric by which we judge performance is the relative square error, RSE, given in dB, and defined,
\begin{equation}\label{eq:RSE}
    RSE = 20\log_{10}\frac{\normf{\euA - \euX^* \starM \euY^*}}{\normf{\euA}}.
\end{equation}
The following experiments were conducted using Matlab, and run on the University of Bath HPC \cite{bath_computing}. The HB120rs v3 partitions were used, which feature 120 AMD EPYC 7V73X CPU cores, 448 GB of RAM and clock frequencies up to 3.5 GHz.

For testing the performance of TASD and TASDII, we restricted ourselves to the t-product framework (with $M$ being the Discrete Fourier Transform), rather than a general $\starM$-product framework. 

When implementing TASD, we use the same stopping conditions as in ASD (see \cref{sec:ASD}). These are (1) $i_\text{stop}$, the maximum number iterations; (2) a tolerance $\epsilon_{con}$ on the relative residual norm, 
\begin{equation}\label{eq:resi}
    res_i = \frac{\normf{\euD - \sample{\euX \starM \euY}}}{\normf{\euD}};
\end{equation}
(3) a tolerance $\epsilon_{ES}$ on the difference in relative residual norm over 50 iterations.
This residual is cheap to compute due to the residual update described in \cref{eq:res update}. Tolerances are typically set to $10^{-4}$ when exact completion is expected, or at the estimated noise level for real data.

The residual update is an effective method to reduce the per-iteration cost of TASD. Unfortunately, other efficiency-improving steps are difficult to implement in MATLAB due to the current lack of support for sparse ND arrays and Einstein notation computations. This slows TASD from its potential maximum efficiency that can be achieved in other implementations. 

On the other hand, when implementing TASDII, we allow LoopedASD to run automatically on each slice. Naturally, the final implementation of ASD after the rank has been estimated, is subject to the same stopping conditions as above.

When completing the x-ray spectromicroscopy data, it should be noted that following the thresholding of singular vales and re-completion of compressed slices, we \emph{do} make use of the additional steps described in \cref{sec:TASDII}, i.e. set slices to zero when they are surrounded by other zero slices. The justification for this approach is described in~\cref{sec:specstruc}, where we discuss the spectral structure of these data sets. 
In short, when using the Fourier transform the outer frontal slices correspond to the more meaningful low-frequency information, and the middle slices correspond to the high-frequency noise. It can be seen that more focus should be put into completing the outer slice, because they contain more significant information and are more amenable to low rank completion. Thresholding SV energy with $\gamma$ will set most noisy slices to 0.
However, in some cases slices are overfit by the completion. To avoid these slices impacting on the quality of completion, we implement the following rule: slices surrounded by 0-slices are also set to 0.

\subsection{Results}\label{sec:Real Data}
Different methods have different parameters to tune:
$(r,p)$ for LoopedASD, $(t,p)$ for TASD and $(\gamma,p)$ for TASDII.
Reliable comparison of those methods requires us to select $r,t$ and $\gamma$
to `optimal' values for each given undersampling ratio $p$. 
In \cref{fig:Real Completion Results}, we vary each parameter independently of the other,
and plot the average RSE~\eqref{eq:RSE} over 20 runs (random realisations of the initial guess).

\begin{figure}[h!]
    \begin{subfigure}{0.32\textwidth}
        \centering
        \includegraphics[width = \textwidth]{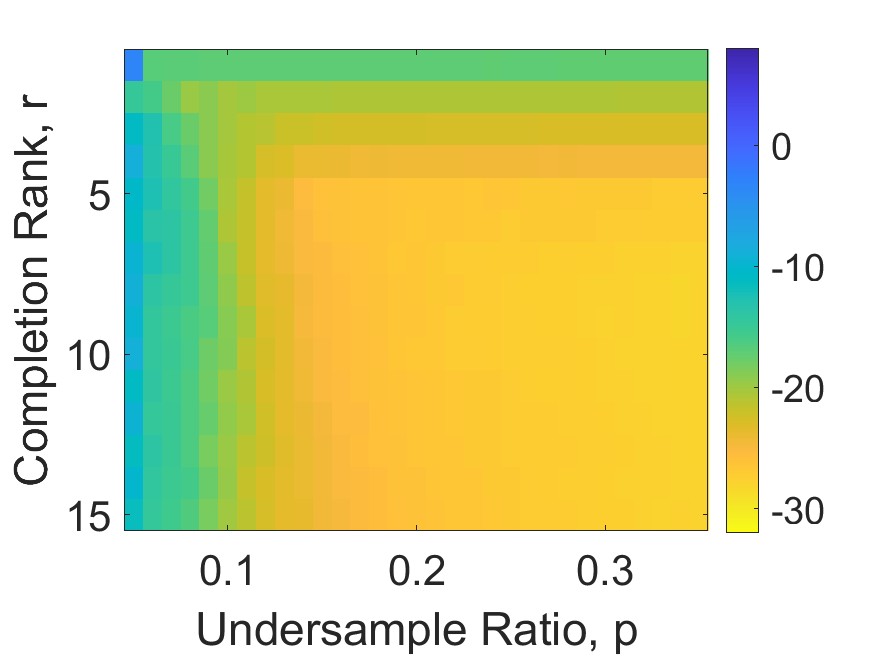}
        \caption{LoopedASD - DS1}
        \label{fig:LoopedASD-DS1}
    \end{subfigure}
    \begin{subfigure}{0.32\textwidth}
        \centering
        \includegraphics[width = \textwidth]{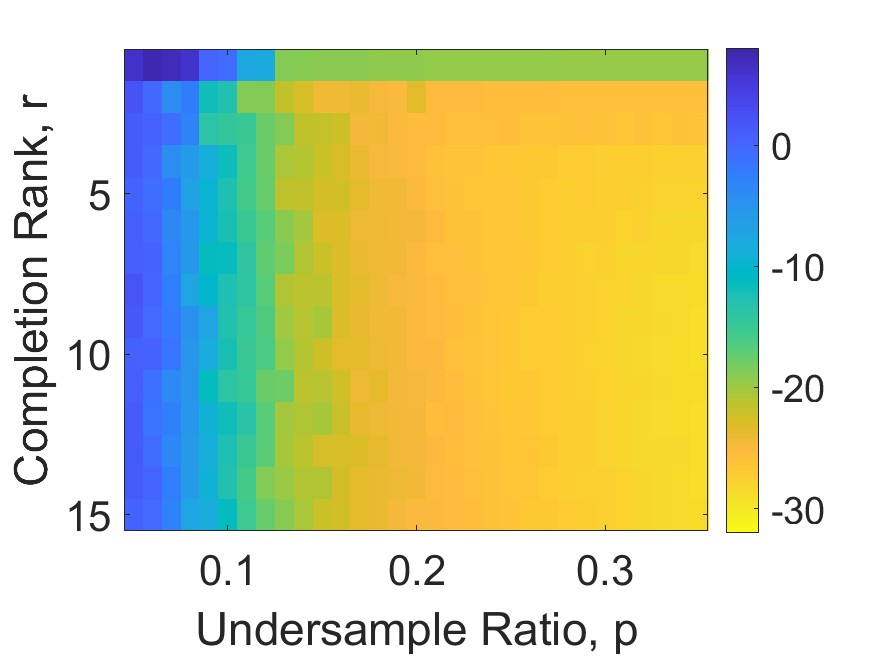}
        \caption{LoopedASD - DS3}
        \label{fig:LoopedASD-DS3}
    \end{subfigure}
    \begin{subfigure}{0.32\textwidth}
        \centering
        \includegraphics[width = \textwidth]{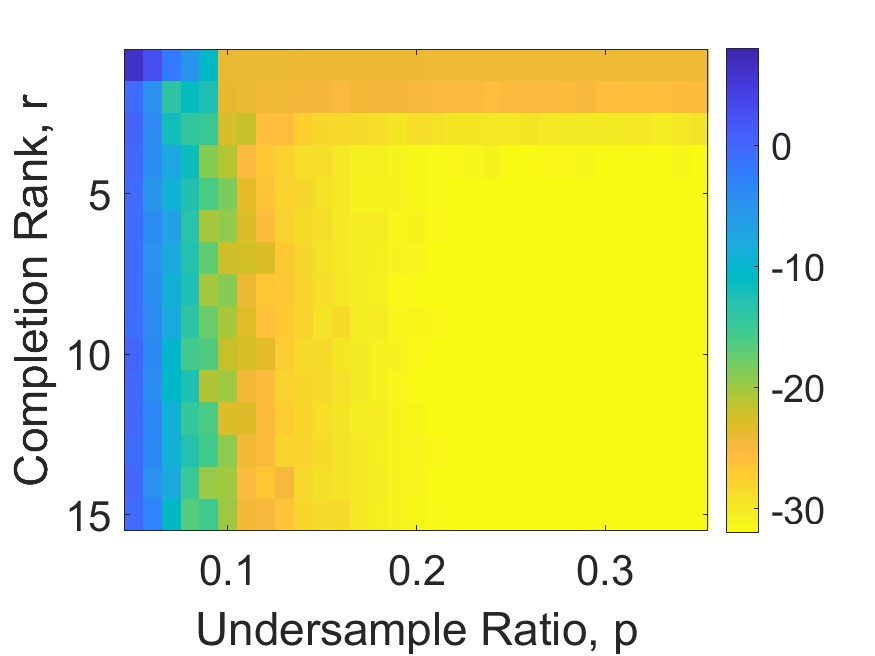}
        \caption{LoopedASD - DS5}
        \label{fig:LoopedASD-DS5}
    \end{subfigure}

    \begin{subfigure}{0.32\textwidth}
        \centering
        \includegraphics[width=\textwidth]{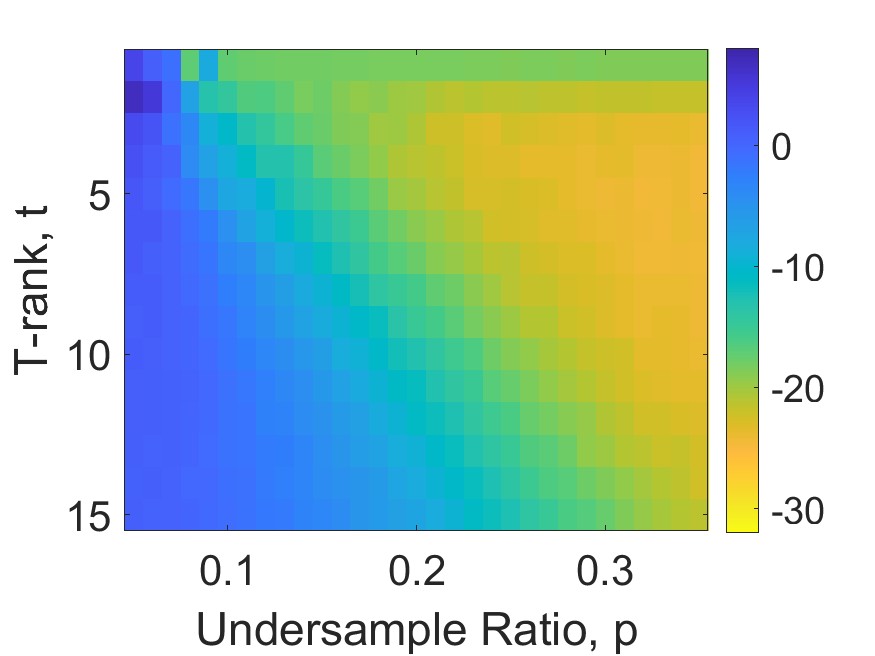}
        \label{fig:TASD-DS1}        
        \caption{TASD - DS1}
    \end{subfigure}
    \begin{subfigure}{0.32\textwidth}
        \centering
        \includegraphics[width=\textwidth]{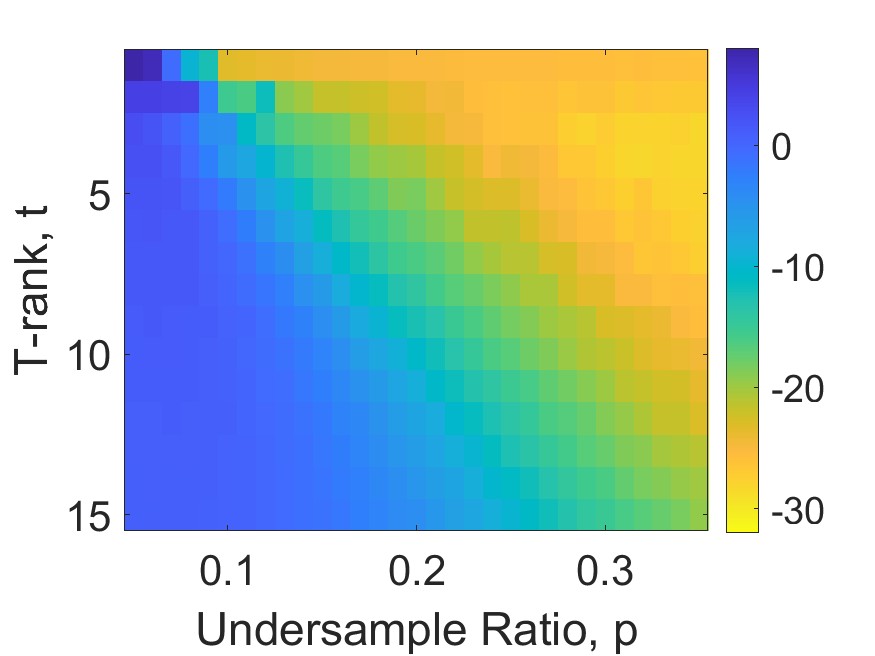}
        \label{fig:TASD-DS3}
        \caption{TASD - DS3}
    \end{subfigure}
    \begin{subfigure}{0.32\textwidth}
        \centering
        \includegraphics[width=\textwidth]{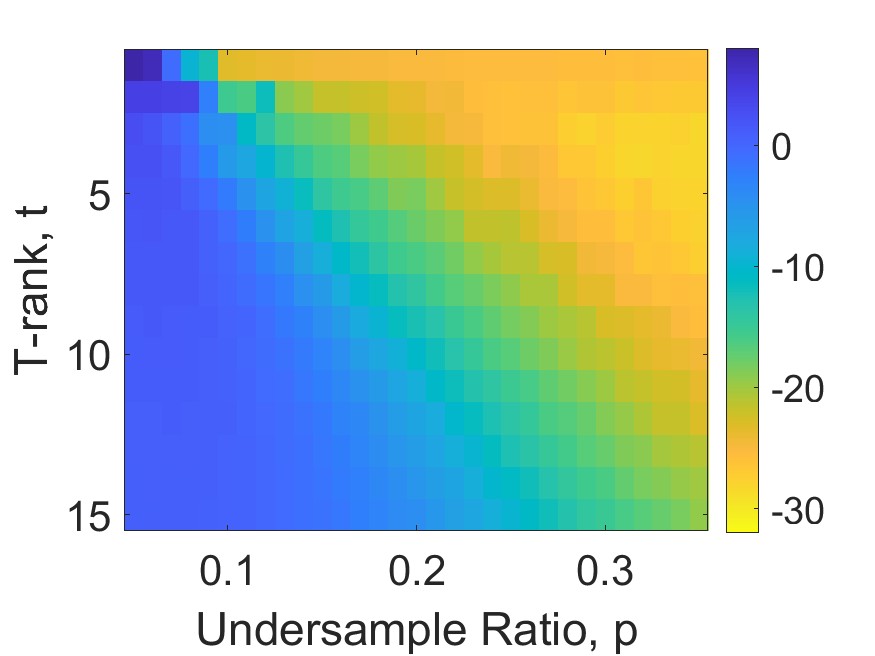}
        \label{fig:TASD-DS5}
        \caption{TASD - DS5}
    \end{subfigure}

    \begin{subfigure}{0.32\textwidth}
        \centering
        \includegraphics[width = \textwidth]{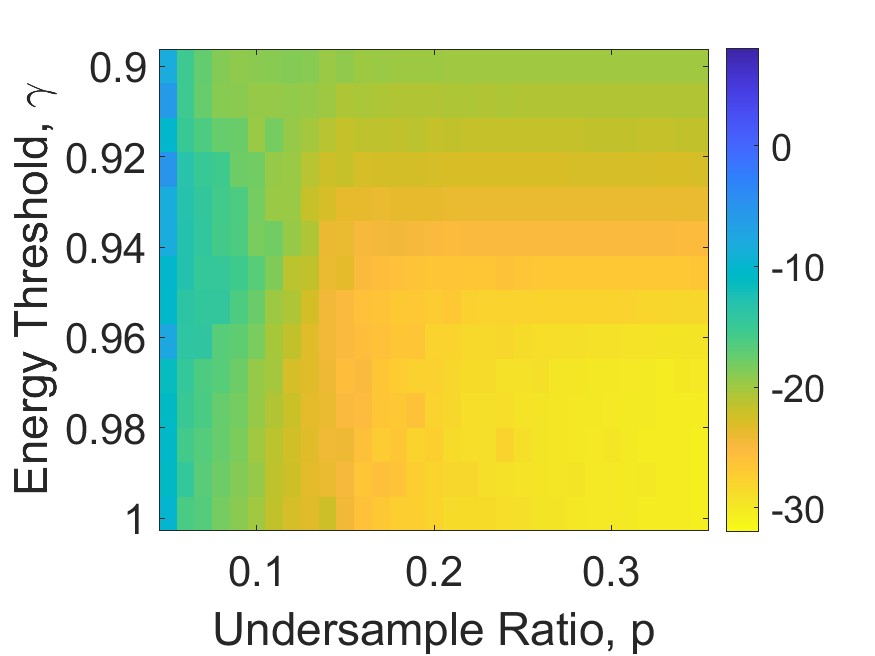}
        \caption{TASDII - DS1}
        \label{fig:TASDII-DS1}
    \end{subfigure}
    \begin{subfigure}{0.32\textwidth}
        \centering
        \includegraphics[width = \textwidth]{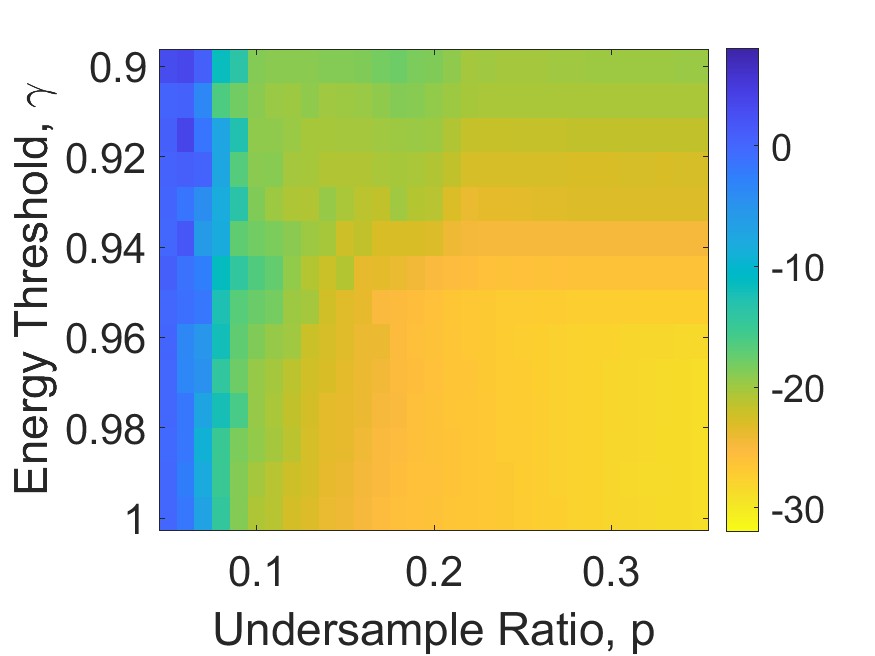}
        \caption{TASDII - DS3}
        \label{fig:TASDII-DS3}
    \end{subfigure}
    \begin{subfigure}{0.32\textwidth}
        \centering
        \includegraphics[width = \textwidth]{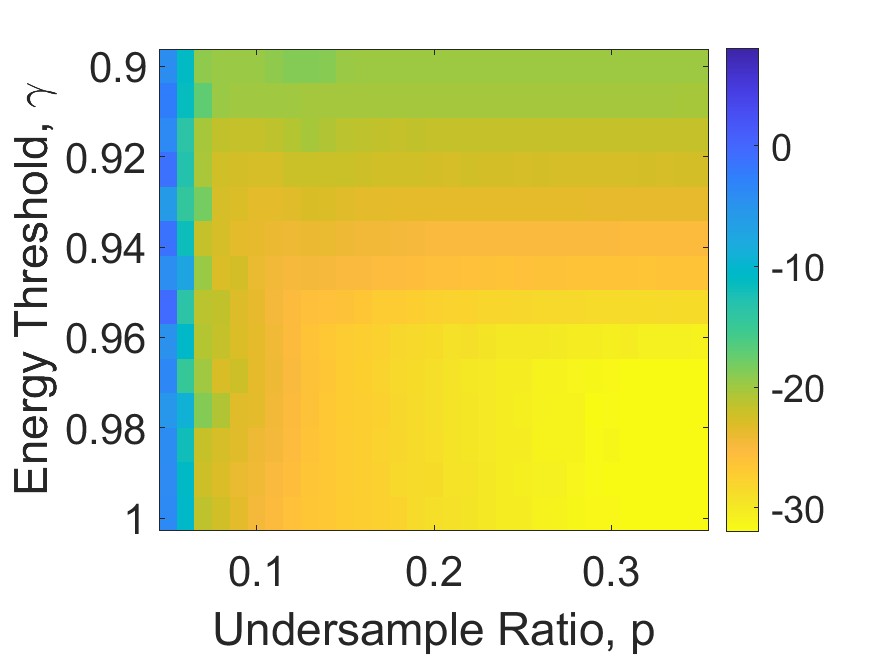}
        \caption{TASDII - DS5}
        \label{fig:TASDII-DS5}
    \end{subfigure}

    \caption{RSE for X-Ray Spectromicroscopy data with different tuning parameters. The top row illustrates the results for LoopedASD, the middle row for TASD, and the bottom row illustrates the results for TASDII. Yellow is better (lower error).}
    \label{fig:Real Completion Results}
\end{figure}

Now for each undersampling ratio $p$, we select the the minimal completion error across the other parameter range to provide our `optimal parameters' $r^*(p),t^*(p)$ and $\gamma^*(p)$ for LoopedASD, TASD and TASDII respectively. Plotting the minimal error against $p$ for each method allows comparison between the different completion algorithms. This can be seen in~\cref{fig:MethodCompare}.

\begin{figure}[h!]
    \begin{subfigure}{0.32\textwidth}
        \centering
        \includegraphics[width=\textwidth]{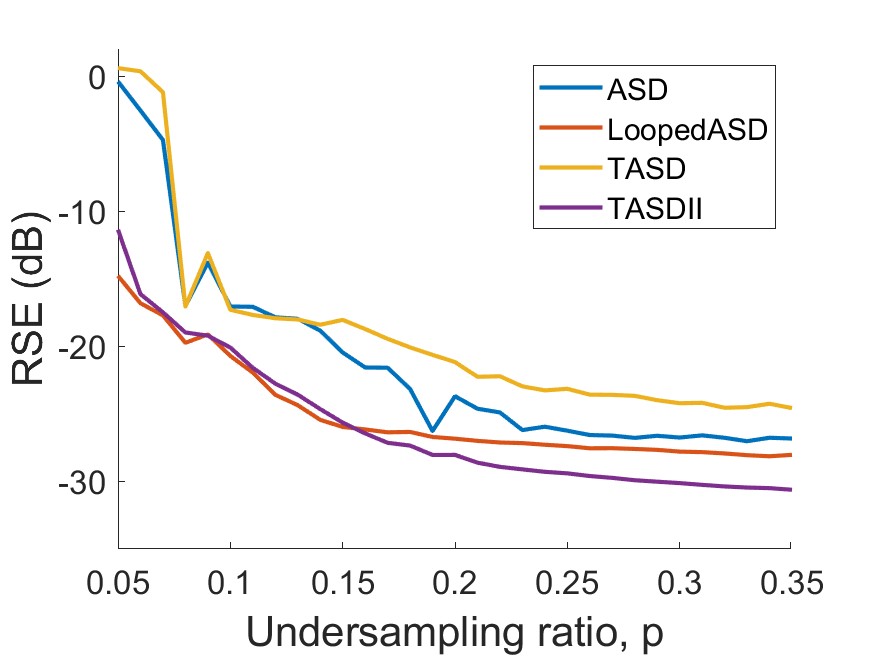}
        \caption{DS1}
        \label{fig:MethodCompareDS1}
    \end{subfigure}
    \begin{subfigure}{0.32\textwidth}
        \centering
        \includegraphics[width=\textwidth]{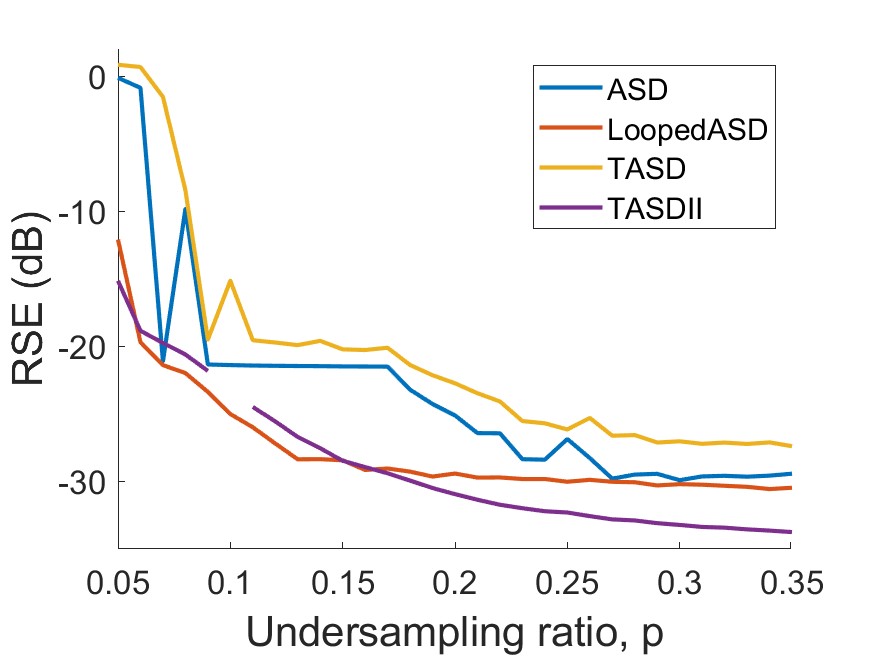}
        \caption{DS2}
        \label{fig:MethodCompareDS2}
    \end{subfigure}
    \begin{subfigure}{0.32\textwidth}
        \centering
        \includegraphics[width=\textwidth]{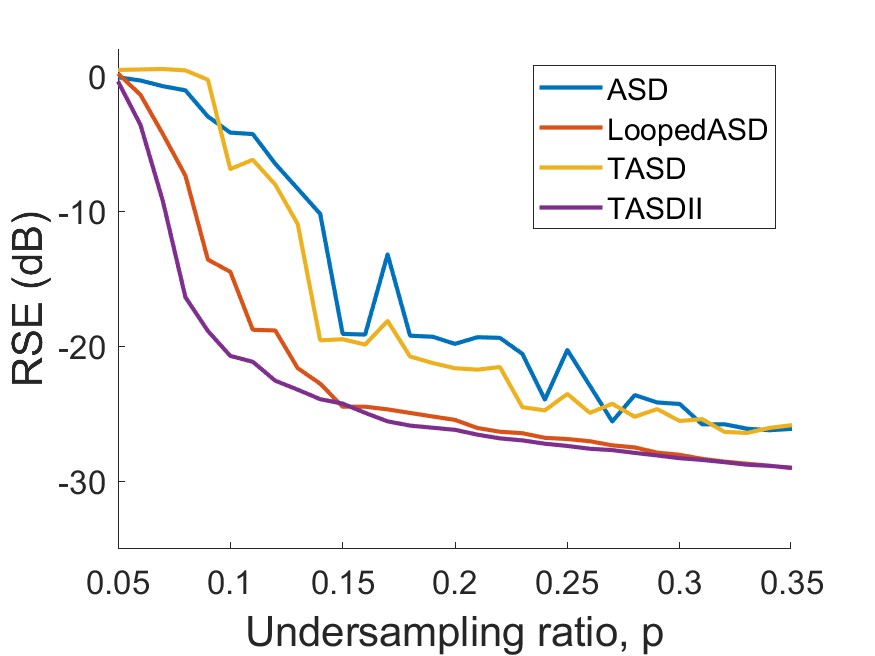}
        \caption{DS3}
        \label{fig:MethodCompareDS3}
    \end{subfigure}

    \begin{subfigure}{0.32\textwidth}
        \centering
        \includegraphics[width=\textwidth]{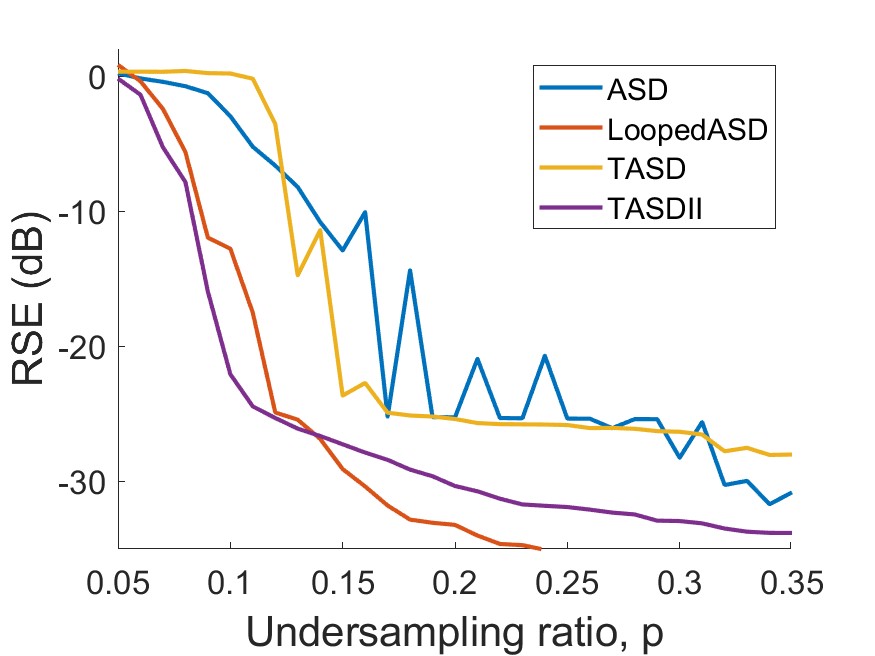}
        \caption{DS4}
        \label{fig:MethodCompareDS4}
    \end{subfigure}
    \begin{subfigure}{0.32\textwidth}
        \centering
        \includegraphics[width=\textwidth]{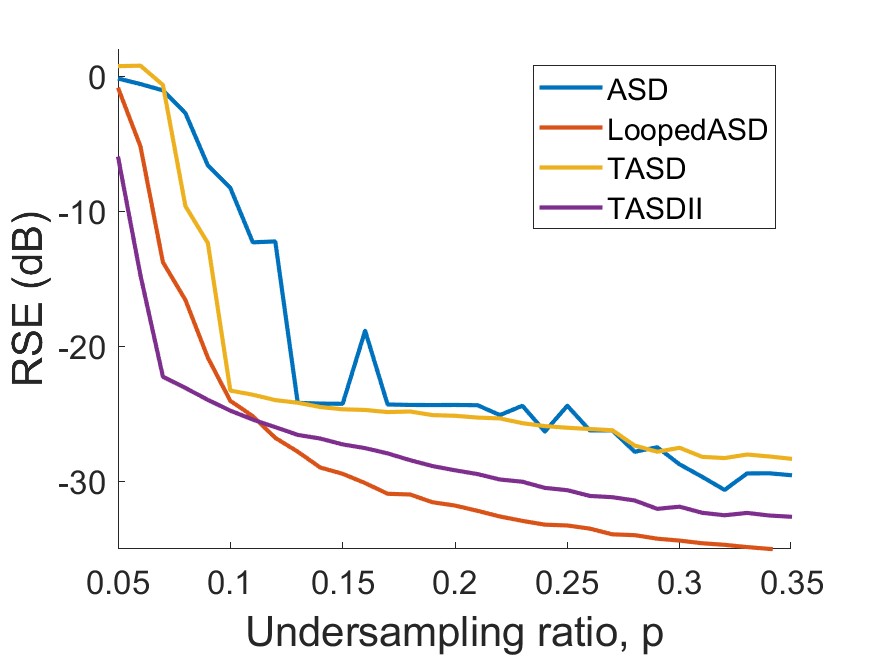}
        \caption{DS5}
        \label{fig:MethodCompareDS5}
    \end{subfigure}
    \begin{subfigure}{0.32\textwidth}
        
    \end{subfigure}
    \caption{Minimal Completion Errors of X-ray Spectromicroscopy data across the parameter space for each undersampling ratio and for each algorithm.}
    \label{fig:MethodCompare}
\end{figure}

It is clear that TASDII performs significantly better than TASD, and is on par with LoopedASD. In \cref{fig:Real Completion Results}, we can see that for lower undersampling ratios (around $p = 0.1$) TASDII is still able to complete the data accurately, with completion errors around -20dB; the results for TASD are much weaker (about -8dB) for the same undersampling ratios. This is equally clear in \cref{fig:MethodCompare}, where we see that TASD is more comparable to ASD, while TASDII produces lower completion errors across all undersampling ratios.

On the other hand, we observe that TASDII is more accurate than LoopedASD for lower undersampling ratios for DS3, DS4 and DS5, and for higher undersampling ratios for DS1 and DS2.

\section{Conclusions and outlook}
\label{sec:conclusions}
We have proposed two new tensor completion algorithms, named TASD and TASDII. We particularly advocate for TASDII, which breaks the tensor completion problem down to many independent matrix completion problems. Indeed, the TASDII performance balances two effects of the shape and structure. On one hand, the shape of each frontal slice of the transformed tensor can be ensured to be more square, with each slice sampled using the Bernoulli distribution. This has been shown to improve how easily a matrix can be completed (i.e. what is the smallest undersampling ratio required). When using an appropriate transform $M$, fewer singular components are required across all slices to recover an accurate reconstruction of the data set. Additionally, the ability to completely parallelise the algorithm means that TASDII is the fastest of the methods examined here, particularly for larger data sets.

On the other hand, each slice of a tensor is smaller than the matrix flattening of the whole tensor. This increases the required undersampling ratio, since a smaller dataset of the same rank requires more known entries to complete. In practice, this balance is problem-dependent.
We can expect the tensor completion to perform better for larger square tensors,
where each slice is both square and sufficiently large, whereas the flattened tensor is significantly rectangular that may degrade the matrix completion accuracy.

There are a number of extension we can envision for the work presented in this paper. (1) Since the Discrete Fourier Transform as matrix $M$ for the $\starM$-product was chosen in this paper for demonstration only, it can be replaced by a more efficient application-oriented transform, for example, the Discrete Cosine Transform.
The latter would also make all tensors in spectromicroscopy real-valued. 
Ultimately, following \cite{optimalM}, $M$ can be learned together with $\euX,\euY$. 
All theoretical results exploited and proved in this paper hold as far as $M$ is a multiple of a unitary matrix. (2) The $\starM$-product decomposition can be extended to higher-dimensional tensors,
for example, adding the 4th dimension of the angle of the specimen rotation.
This may also involve more structured sampling patterns, such as rastering over two latter dimensions. (3) Further regularizations are possible within the tensor completion problem, such as penalising the norm of finite differences or total variation of $\euX$ and $\euY$ to ensure the smoothness of the completion.
Both TASD and TASDII should remain the same (most importantly in terms of the rank adaptation loop) up to technical modifications in the gradient.

\appendix

\section{Algorithmic details about LoopedASD}\label{sect: loopedASD}
The main steps of the LoopedASD algorithm are outlined in \cref{algo: loopedASD}, where we can clearly see that the backbone of LoopedASD is a rank incremental loop. 
\begin{algorithm}
        \caption{LoopedASD for matrix completion with automatic rank estimation}\label{algo: loopedASD}
            \hspace*{\algorithmicindent} \textbf{Input:} samples $D\in\C^{n_1\times N}$, partitioned indeces $\Omega=\bigcup_{\ell=1}^k\Omega_{\ell}$, $r_{\max}$\\
\hspace*{\algorithmicindent} \textbf{Output:}  $X^*$, $Y^*$
\begin{algorithmic}[1]
            \STATE Take $X^\ast=[\;]$, $Y^\ast=[\;]$
            \FOR {$j = 1,\dots,r_{\max}$}
            \STATE Choose $i\in [k]$ uniformly at random and take $\Omega_{\text{train}}=\Omega\setminus\Omega_i$, $\Omega_{\text{test}}=\Omega_i$ 
            \STATE Take $X_0=[X^\ast, x]\in\C^{n_1\times j}$, with $x\in\C^{n_1}$ as in \cref{loopedASDrescale}
            \STATE Take $Y_0=[(Y^\ast)^{\rm H},y]^{\rm H}\in\C^{j\times N}$, with $y\in\C^{N}$ as in \cref{loopedASDrescale}
            \STATE run ASD \cref{eq:ASD} on $\Omega_{\text{train}}$, with $r=j$ and inputs $X_0,\,Y_0$, to get $X^*,\,Y^*$
            \STATE compute the test error $t_j:=\|\mathcal{P}_{\Omega_{\text{test}}}(D)-\mathcal{P}_{\Omega_{\text{test}}}(X^\ast Y^\ast)\|_F$
            \ENDFOR
            \STATE Determine the completion rank $r$ from $t_j$, $j=1,\dots,r_{\max}$.
            \STATE Compute $[U,S,V]=\text{\texttt{svd}}(X^\ast Y^\ast)$
            \STATE Take $X_0=U(:,1:r)$, $Y_0=S(1:r,1:r)V(:,1:r)^{\rm H}$
            \STATE run ASD \cref{eq:ASD} on $\Omega$, with inputs $X_0,\,Y_0$, to get $X^*$ and $Y^*$
        \end{algorithmic}
\end{algorithm}

The main purpose of the loop is to estimate the ($\epsilon$-)rank of the data through the cross validation approach detailed in lines 3, 6 and 7 (i.e., intermediate ASD runs are on training data, while the residual (or test error) is evaluated on the test data). Indeed, if the test error at the $j$th loop is close to the best rank $j$ approximation error of the full matrix for all $j=1,\dots,r_{\max}$, and if the singular values of the latter quickly decay, the Eckart-Young-Mirsky theorem \cite{E-Y-M_Theorem} implies that the behavior of both the test errors and the singular values are similar. Thus any rank detection strategy available in the literature can be applied directly to the test errors at step 9 of \cref{algo: loopedASD}. Here, also following the standard spectromicroscopy practice \cite{LEROTIC200435}, we use KNEEDLE \cite{kneedle} to detect the point of maximal curvature in the sequence of singular values; alternatively one can threshold the numerical gradients of the singular values. 

Once the completion rank has been identified, the factors from the last looped completion are projected down to this rank and used as initial guesses for one last ASD run to output the complete data (lines 10 to 12 of \cref{algo: loopedASD}). Within the increasing rank loop, the initial guesses for the $j$th ASD call are obtained by augmenting the factors completed at the $(j-1)$th ASD call (empty matrices if $j=1$) with rescaled Gaussian random vectors $g_x$ and $g_y$, such that, in lines 4 and 5 of \cref{algo: loopedASD},
\begin{equation}\label{loopedASDrescale}
x = 2^{-j}\sqrt{p}\|D\|_F^{-1}\|g_x\|_F^{-1}g_x,\quad
y = 2^{-j}\sqrt{p}\|D\|_F^{-1}\|g_y\|_F^{-1}g_y,   
\end{equation}
where $p$ is the undersampling ratio. 

It is experimentally observed that the looping structure of \cref{algo: loopedASD} significantly improves the probability of completion of sparse data even with a given fixed rank $r_{\text{fix}}$, especially when the condition number of the data matrix is large. In this setting, one applies LoopedASD by performing only $r_{\text{fix}}$ intermediate ASD calls (lines 2 to 8 of \cref{algo: loopedASD}), proceeding then immediately to the last full completion (lines 10 to 12 of \cref{algo: loopedASD}). Heuristically, the looping strategy helps avoiding convergence of ASD to spurious local minima by setting favourable initial guesses for the completion at each rank, since completing a lower rank subspace of higher rank data can be typically done accurately. 

\section{Further results for $\starM$ framework} 

Recall the definition of the tensor t-trace for $\euA \in \C^{n_1\x n_1\x n_3}$,
    \begin{equation}
        \tTr(A) = \sum^{n_3}_{k=1} tr(\hat{\euA}_{:,:,k}),
    \end{equation}
where \textit{tr} is the matrix trace. Using the definition of the $\starM$ product, it is simple to show the t-trace satisfies the following properties,
\begin{align}
    \tTr(\euA + \euB) &= \tTr(\euA) + \tTr(\euB)\\
    \tTr(\lambda\euA) &= \lambda \tTr(\euA)\\
    \tTr(\euA^T) &= \tTr(\euA)\\
    \tTr(\euA \starM \euB) &= \tTr(\euB \starM \euA).
\end{align}
In \cref{lem:TIP}, we stated that the TIP is an inner product, and that, for unitary $M$,
    \begin{equation}
        \langle \euA, \euA \rangle = \normf{\euA}^2.
    \end{equation}
\begin{proof}
    The necessary properties of an inner product follow from the definition of the TIP and the properties of the tensor transpose and the tensor trace. Thus for suitable tensors $\euA,\ \euB,\ \euC$ and scalars $\lambda$ and $\mu$, we have
    \begin{align}
        \overline{\langle \euA,\euB\rangle} =  &=\langle \euB,\euA \rangle \\
        \langle \lambda \euA + \mu \euB,\euC\rangle &= \lambda \langle\euA,\euC\rangle + \mu \langle \euB,\euC\rangle\\
        \langle \euA,\euA\rangle &\geq 0 \\
        \langle \euA,\euA\rangle &= 0  \quad \Leftrightarrow \euA = 0
    \end{align}

    Indeed, the last two lines are equivalent to proving that $\langle \euA,\euA \rangle = \normf{\euA}^2$ for unitary $M$. We can write the inner product as,
    \begin{align*}
        \langle \euA,\euA \rangle &= \tTr(\euA \starM \euA^{\rm H})
        = \sum_{k=1}^{n_3} \text{tr}(\hat{A}_{:,:,k} \hat{A}^{\rm H}_{:,:,k}) = \sum_{k=1}^{n_3} \sum_{ij}(\hat{A}_{i,j,k} \hat{A}^{\rm H}_{j,i,k}) \\
        &= \sum_{k=1}^{n_3} \sum_{ij}(\hat{A}_{i,j,k} \overline{\hat{A}}_{i,j,k}) \\
        &= \sum_{ijk} |\hat{A}_{i,j,k}|^2 = \normf{\hat{\euA}}
    \end{align*}
    If we assume $M$ is unitary, then we get
    \begin{equation}
        \normf{\hat{\euA}} = \normf{\euA \x_3 M} = \normf{M\euA_{(3)}} = \normf{\euA_{(3)}} = \normf{\euA}
    \end{equation}    
    In the case that $M$ is a scalar multiple of a unitary matrix, then one would need to appropriately scale the inner product to ensure this identity holds.
\end{proof}

\section{TASD}\label{sec:TASDderiv}
We now derive the gradients and step sizes given in equations \cref{eq:TASDgrad} and \cref{eq:stepsizes} for the TASD. To compute these values, we make use of the properties of the tensor inner product found in \cref{lem:TIP} and the following result. The proofs of these lemma are provided in the supplemental document.
\begin{lemma}
   Let $\mathcal{P}$ be an orthogonal projection of tensors, so that \linebreak[4]$\mathcal{P}:\C^{n_1\x n_2\x n_3} \rightarrow \C^{n_1\x n_2\x n_3}$. Then,
   \begin{align}\label{IP_projection}
       \langle \mathcal{P}(\euA), \euB \rangle = \langle \mathcal{P}(\euA), \mathcal{P}(\euB) \rangle = \langle \euA, \mathcal{P}(\euB) \rangle
   \end{align}
\end{lemma}

To ensure the sparse projections are more straightforward to understand in the following computations, we can express the sampling operator using the indicator tensor $\mathds{1}_{\Omega} \in \{0,1\}^{n_1 \x n_2 \x n_3}$, defined
\begin{equation}
    [\mathds{1}_{\Omega}]_{ijk} = 
    \begin{cases}
        1 & \text{if} \quad (i,j,k) \in \Omega,\\
        0 & \text{otherwise}.
    \end{cases}
\end{equation}
The projection itself is then computed by taking the Hadamard (pointwise) product with, i.e. $\sample{A} = \mathds{1}_{\Omega} \circ A$
\begin{lemma}
    The gradients of the objective function (see \cref{eq:objfunc}) with one fixed parameter, $f_{\euY} (\euX)$ and $f_{\euX} (\euY)$, are given by, 
    \begin{equation}
        \nabla f_{\euY} (\euX) = -(\euD - \sample{\euX \starM \euY })\starM\euY^{\rm H}, \qquad \nabla f_{\euX} (\euY) = -\euX^{\rm H} \starM (\euD - \sample{\euX \starM \euY }).
    \end{equation}
\end{lemma}

\begin{proof}
    We derive the gradient for $f_{\euY} (\euX)$ only, the derivation for $f_{\euX} (\euY)$ is similar. \cref{def:Mprod} outlines the $\starM$-product of tensors $\euX \in \C^{n_1\x n_2\x n_3}$ and $\euY \in \C^{n_2\x n_4\x n_3}$, which can be expressed in terms of the entries of $\euX$, $\euY$, and $M$ as
    \begin{align}\label{eq:t-prod index}
        \hat{\euX}_{:,:,\gamma} = \sum_{s=1}^{n_3} M_{\gamma s}\euX_{:,:,s}& \qquad \text{and} \qquad \hat{\euY}_{:,:,\gamma} = \sum_{s=1}^{n_3} M_{\gamma s}\euY_{:,:,s}.\\
        (\euX \starM \euY)_{\alpha,\beta,\gamma} &= \sum_{\gamma'=1}^{n_3} M_{\gamma\gamma'}^{-1}\; \sum_{l=1}^{n_2} \hat{\euX}_{\alpha,l,\gamma'}\; \hat{\euY}_{l,\beta,\gamma'},\\
        &= \sum_{\gamma' l} M^{-1}_{\gamma \gamma'}\; \hat{\euX}_{\alpha,l,\gamma'}\; \hat{\euY}_{l,\beta,\gamma'},
    \end{align}
    where we will now omit the limits of the sums for clarity. We now compute the non zero partial derivatives of the $\starM$-product with respect to the elements of $\euX$. Recall that we have assumed $M = cW$ for unitary matrix $W$ and non-zero $c$, so that $M^{-1} = \frac{1}{c}W^{\rm H}$. For clarity, we have omitted the scalar constant $c$ from the derivation.
    \begin{equation}
        \begin{aligned}\label{eq:tprodx}
            \pdiff{(\euX \starM \euY)_{\alpha \beta \gamma}}{\euX_{ijk}} &= \sum_{\gamma' l} M^{\rm H}_{\gamma \gamma'}\; \pdiff{\hat{\euX}_{\alpha,l,\gamma'}}{\euX_{ijk}}\;\hat{\euY}_{l,\beta,\gamma'}\\
            &= \sum_{\gamma' l} M^{\rm H}_{\gamma \gamma'}\; \hat{\euY}_{l,\beta,\gamma'} \sum_{s}M_{\gamma' s}\pdiff{\euX_{\alpha,l,s}}{\euX_{ijk}}\\
            &= \sum_{\gamma' l} M^{\rm H}_{\gamma \gamma'}\; \hat{\euY}_{l,\beta,\gamma'} \sum_{s}M_{\gamma' s}\delta_{i\alpha} \delta_{jl} \delta_{ks}\\
            &= \delta_{i\alpha} \sum_{\gamma'} M^{\rm H}_{\gamma \gamma'}\; \hat{\euY}_{j,\beta,\gamma'} M_{\gamma' k}\\
        \end{aligned}
    \end{equation}
    where we have used the kronecker delta in the $3^{rd}$ line. Similarly, by differentiating with respect to $\euY_{ijk}$ gives
    \begin{align}
        \pdiff{(\euX \starM \euY)_{\alpha \beta \gamma}}{\euY_{ijk}}\ =\  \delta_{j\beta} \sum_{\gamma'} M^{\rm H}_{\gamma\gamma'}\; M_{\gamma' k}\; \hat{\euX}_{\alpha i \gamma'}
    \end{align}
    
    Finally, we compute the gradient, $\nabla f_{\euY}(\euX)$, by computing the $(ijk)-th$ partial derivatives. To simplify the expressions, we use the substitution for the residual $\euR = \euD - \sample{\euX\starM\euY}$.
    \begin{equation}
        \begin{aligned}
            \left[\nabla f_{\euY}(\euX)\right]_{ijk} &= \pdiff{ }{\euX_{ijk}} \frac{1}{2}\normf{\euD - \sample{\euX\starM\euY}}^2\\ 
            &= \sum_{\alpha \beta \gamma} \euR_{\alpha \beta \gamma} \overline{ \pdiff{\left(\euD - \mathds{1}_{\Omega} \circ (\euX \starM \euY)\right)_{\alpha \beta \gamma}}{\euX_{ijk}}}\\
            &= \sum_{\alpha \beta \gamma}\euR_{\alpha \beta \gamma}\cdot -(\mathds{1}_{\Omega})_{\alpha \beta \gamma}\cdot \overline{\pdiff{\left(\euX \starM \euY)\right)_{\alpha \beta \gamma}}{\euX_{ijk}}}\\
            &= -\sum_{\alpha \beta \gamma}\euR_{\alpha \beta \gamma}\ \cdot\  \delta_{i\alpha} \overline{\sum_{\gamma'} M^{\rm H}_{\gamma \gamma'} M_{\gamma' k} \hat{\euY}_{j\beta \gamma'}}\\
            &= -\sum_{\beta \gamma \gamma'} \overline{M}^{\rm H}_{\gamma \gamma'}\; \overline{M}_{\gamma' k}\; \euR_{i \beta \gamma}\; \overline{\hat{\euY}}_{j \beta \gamma'}\\
            &= -\sum_{\beta \gamma \gamma'} M^{\rm H}_{k \gamma'}\; M_{\gamma' \gamma}\; \euR_{i \beta \gamma}\; \hat{\euY}^{\rm H}_{\beta j \gamma'} \\
            &=-\sum_{\beta \gamma'} M^{\rm H}_{k \gamma'}\; \hat{\euR}_{i \beta \gamma'}\; \hat{\euY}^{\rm H}_{\beta j \gamma'} \\
            &=-\sum_{\gamma' \beta} M^{-1}_{k \gamma'}\;  \hat{\euR}_{i\beta \gamma'} \overline{\hat{\euY}}_{j \beta \gamma'} \\
            &=-(\euR \starM \euY^{\rm H})_{ijk}\quad =\quad -((\euD - \sample{\euX \starM \euY}) \starM \euY^{\rm H})_{ijk}.
        \end{aligned}
    \end{equation}  
    where we have used that $M$ is a unitary matrix. Note that the residual $\euR$ has already been projected by $\sampleOP$.
\end{proof}
We now derive the step sizes used for TASD.
\begin{lemma}
    Let $\eta_x$ and $\eta_y$ be the exact steepest descent step sizes for the directions $-\nabla f_{\euY}(\euX)$ and $-\nabla f_{\euX}(\euY)$ respectively. Then,
    \begin{align}\label{eq:stepsizes2}
            \eta_{\euX} &= \frac{\normf{\nabla f_{\euY} (\euX)}^2}{\normf{\sample{\nabla f_{\euY} (\euX) \starM \euY}}^2} &\eta_{\euY} = \frac{\normf{\nabla f_{\euX}(\euY)}^2}{\normf{\sample{\euX \starM \nabla f_{\euX}(\euY)}}^2}.
    \end{align}
\end{lemma}
\begin{proof}
    We only provide the proof for $\eta_{\euX}$, the proof for $\eta_{\euY}$ is similar. Let $g:\R \rightarrow \R$ describe the square error of the iterates as we move along the step direction. We wish to minimise this function over $\eta$,
    \begin{equation}
        \eta_{\euX} = \argmin_{\eta}{g(\eta)}\quad \text{where} \quad  g(\eta) = \normf{\euD - \sample{(\euX - \eta\nabla f_{\euY}(\euX)) \starM \euY}}^2.
    \end{equation}
    To solve this, we simply set the derivative to 0 and compute $g'(\eta_{\euX}) = 0$. By writing $g$ as an inner product, see definition \ref{def:TIP}, we have that 
    \begin{equation}
        g'(\eta_{\euX}) = \big\langle \euD -\sample{(\euX-\eta_{\euX} \nabla f_{\euY}(\euX))\starM\euY} , \sample{\eta_x \nabla f_{\euY} (\euX)\starM\euY }\big\rangle.
    \end{equation}
    Thus, we compute,
    \begin{align*}
        g'(\eta_{\euX}) &= 0\\
        g'(\eta_{\euX}) &= \big \langle \euD - \sample{\euX\starM\euY)},\sample{\nabla f_{\euY}(\euX)\starM\euY} \big \rangle + \big \langle \sample{\eta_{\euX} \nabla f_{\euY}(\euX) \starM \euY)},\sample{\nabla f_{\euY}(\euX)\starM\euY} \big \rangle \\
        &= \big \langle (\euD - \sample{\euX\starM\euY)})\starM\euY^{\rm H},\nabla f_{\euY}(\euX) \big \rangle + \eta_{\euX} \big \langle \sample{\nabla f_{\euY}(\euX) \starM \euY)},\sample{\nabla f_{\euY}(\euX)\starM\euY} \big \rangle \\
        &= \big \langle - \nabla f_{\euY}(\euX) ,\nabla f_{\euY}(\euX) \big \rangle + \eta_{\euX} \normf{\sample{\nabla f_{\euY}(\euX) \starM \euY)}}^2\\
        &= - \normf{\nabla f_{\euY}(\euX)}^2 + \eta_{\euX} \normf{\sample{\nabla f_{\euY}(\euX) \starM \euY)}}^2\\
    \end{align*}
    Where we have used the linearity of the t-product and the distribution of orthogonal projections over the TIP. Note that in the third line, the left expression of the first inner product is the definition of the gradient. By rearranging the final equation, we get the desired expression for $\eta_{\euX}$.    
\end{proof}

\section{Spectral Structure of x-ray spectromicroscopy data}\label{sec:specstruc}

We use our knowledge of the spectral structure of x-ray spectromicroscopy to customise TASDII and improve the completion results. This is simply an example of the choices that can be made, and is in no way prescriptive. Further study of other data sets may yield other patterns that can be taken advantage off, whether using different transforms or focusing on completing different frontal slices.

\begin{figure}[htb]
    \begin{subfigure}{0.4\textwidth}
        \centering  
        \includegraphics[width = \textwidth]{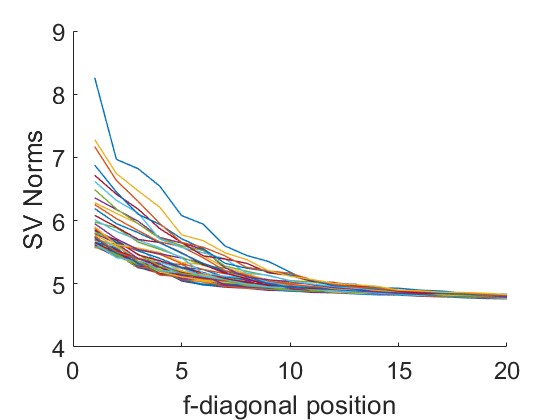}
        \caption{SVs of all slices}
        \label{fig:DS1SVall}
    \end{subfigure}
    \begin{subfigure}{0.4\textwidth}
        \centering  
        \includegraphics[width = \textwidth]{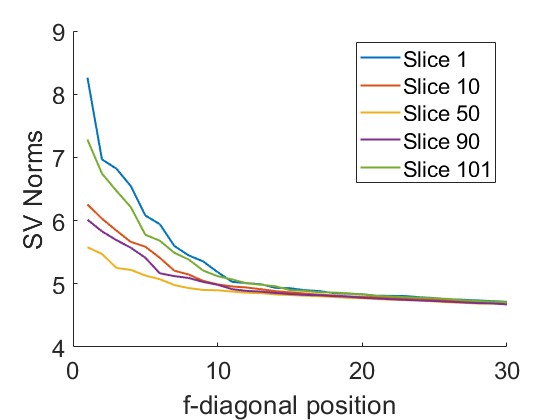}
        \caption{SVs of select slices}
        \label{fig:DS1SVsome}
    \end{subfigure}
    
    \begin{subfigure}{0.4\textwidth}
        \centering  
        \includegraphics[width = \textwidth]{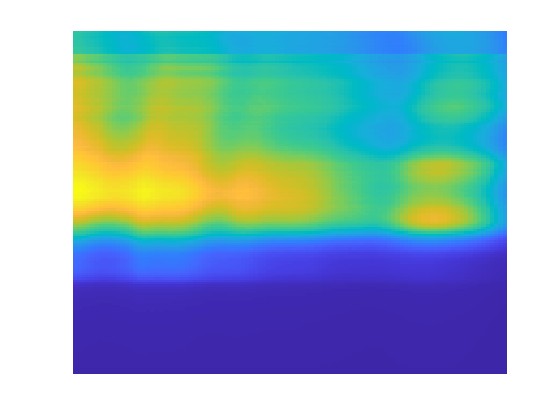}
        \caption{Image of frontal slice 1, $\hat{A}_{:,:,1}$}
        \label{fig:DS1hat1}
    \end{subfigure}
    \begin{subfigure}{0.4\textwidth}
        \centering  
        \includegraphics[width = \textwidth]{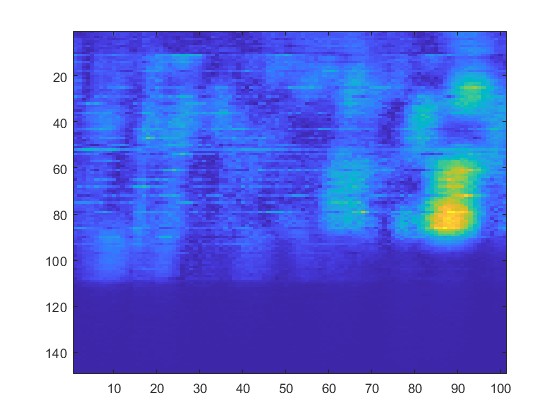}
        \caption{Image of frontal slice 50, $\hat{A}_{:,:,50}$}
        \label{fig:DS1hat50}
    \end{subfigure}

    \caption{First singular values of all slices of the transformed tensor DS1 (a), singular values of selected slices of the same tensor (b), the first slice $\hat{\euA}(:,:,1)$ (c), and the fiftieth slice $\hat{\euA}(:,:,50)$ (d) of the transformed DS1.}
    \label{fig:DS1SVimages}
\end{figure}

The spectral breakdown of the dataset DS1 is shown in Figure~\ref{fig:DS1SVimages}.
We see that following the Fourier transform most of the spectral information is concentrated in only a few slices. This is evident as only a handful of slices contain singular values greater than $10^6$, and only 5 slices contain singular values greater than $10^7$. These more significant slices still decay quickly to the relative noise level within the first 10 positions, while the vast majority of slices show very little variation in the singular values. In Figure~\ref{fig:DS1SVsome}, we have isolated a smaller number of slices to demonstrate that it is the outer slices that contain spectral information, while the middle slices are less significant.

Since we're using the Fourier transform, the outer slices describe the low frequency components of the data, while middle slices contain high frequency components, mostly noise. This can be seen in Figures~\ref{fig:DS1hat1} and Figure~\ref{fig:DS1hat50}, where we have plotted the first and fiftieth transformed frontal slices of DS1. It is clear that only the outer slices contain useful information.

From a completion perspective, the outer slices are also more amenable to low rank completion. Due to the decay in Singular values, these slices are fairly ill conditioned, however in this case we can consider this something of a good thing. The decay and ill-conditioning mean that most of the variation of these slices are contained in only a few components, so that low rank approximations exist with very small error. This opens the door for more accurate low rank completions (the ill conditioning will be dealt with by LoopedASD).

Contrary to this, the middle noisy slices are better conditioned, but with such little variation in the singular values, the accuracy of low rank completions is bounded by the inaccuracy of the optimal low rank approximation (according to the Eckart-Young theorem). In fact, during preliminary testing we found that the middle slice would not complete at all and in many cases, the relative error of certain slices was greater than 1. Cutting these slices out by setting them to zero improved the completion overall. 

To avoid any such slices not being set to zero, we ensured that slices surrounded by zero slices were also set to zero. This removed any spurious, overfit components, but did no affect the more significant outer slices.

\section*{Acknowledgments}
OT is partially supported by a scholarship from the EPSRC Centre for Doctoral Training
in Statistical Applied Mathematics at Bath (SAMBa), under the project EP/S022945/1. 
OT, SD and SG acknowledge Diamond Light Source for time on Beamline/Lab I14 under Proposal MG31039, and Paul Quinn for discussions about X-ray spectromicroscopy applications. OT, SD and SG gratefully acknowledge the University of Bath’s Research Computing Group \cite{bath_computing} for their support in this work. OT, SG and MK acknowledge the Isaac Newton Institute for Mathematical Sciences, Cambridge, for the support and hospitality during the programme ``Rich and Nonlinear Tomography - a multidisciplinary approach" (supported by EPSRC grant no EP/R014604/), where initial work on this paper was undertaken. 

\bibliographystyle{siamplain}
\bibliography{references}
\end{document}

%% file: ex_shared.tex
\usepackage{lipsum}
\usepackage{amsfonts}
\usepackage{graphicx}
\usepackage{epstopdf}
\usepackage{algorithmic}
\usepackage{bbm}
\ifpdf
  \DeclareGraphicsExtensions{.eps,.pdf,.png,.jpg}
\else
  \DeclareGraphicsExtensions{.eps}
\fi

\newsiamremark{remark}{Remark}
\newsiamremark{hypothesis}{Hypothesis}
\crefname{hypothesis}{Hypothesis}{Hypotheses}
\newsiamthm{claim}{Claim}

\headers{Alternating steepest descent for tensor completion}{O. Townsend, S. Dolgov, S. Gazzola and M. Kilmer}

\title{Alternating steepest descent methods for tensor completion with applications to spectromicroscopy
}

\author{Oliver Townsend\thanks{Department of Mathematical Sciences, University of Bath, UK 
  (\email{ot280@bath.ac.uk}, \email{sd901@bath.ac.uk}, \email{sg9680@bath.ac.uk}).}
\and Sergey Dolgov\footnotemark[2]
\and Silvia Gazzola\footnotemark[2]
\and Misha Kilmer\thanks{Department of Mathematics, Tufts University, Medford, MA 02155 (\email{misha.kilmer@tufts.edu}).}}

\usepackage{amsopn}
